\documentclass[fleqn]{article}

\usepackage{amsthm}
\usepackage{amssymb}
\usepackage{amsmath}
\usepackage{graphicx}
\usepackage{eufrak} 
\usepackage{comment}
\usepackage{mathrsfs}
\usepackage{eucal}  
\usepackage{mathrsfs}
\usepackage{mathtools}
\usepackage{ifthen}
\newboolean{skript}
\newcommand{\beweis}{\ifthenelse{\boolean{skript}} {
\begin{proof}[Beweis]
\mbox{}\vfill \mbox{}
\end{proof}
\pagebreak }{} }

\allowdisplaybreaks[3]

\providecommand{\Fourier}{\mathcal{F}}

\usepackage{multirow}
\usepackage{multicol}

\let\scr\mathscr     

\renewcommand{\hat}{\widehat}
\renewcommand{\tilde}{\widetilde}

\usepackage{titlesec}

\usepackage{textfit}

\providecommand{\Infourier}{\mathcal{F}^{-1}  }

\usepackage{dsfont}
\usepackage{geometry}
\newcommand{\ebig}[1]{\mathbb{E}\big[  #1  \big]  }
\newcommand{\Ebig}[1]{\mathbb{E}\Big[  #1  \Big]  }

\newcommand{\ebigg}[1]{\mathbb{E}\bigg[  #1  \bigg]  }

\providecommand{\intu}{\int_0^u }

\input{mymath_art.sty}

\providecommand{\DMax}{\Delta_{\text{\tiny{Max} } }}

\providecommand{\inth}{\int_{-1/h}^{1/h}  }

\providecommand{\lintu}{\intu\limits}

\providecommand{\epx}{\hat{p}(x)  }

\usepackage[round,authoryear]{natbib}

\providecommand{\Levy}{L\'evy}
\mathtoolsset{showonlyrefs}

\providecommand{\Dmax}{\Delta_{\tiny{max}} }

\providecommand{\m}{\text{-}}

\providecommand{\Rem}{\textrm{R}}

\providecommand{\ind}[1]{\mathds{1}_{#1}}
\begin{document}
\title{\Large
Nonparametric estimation for irregularly sampled L\'evy processes
}
\author{Johanna Kappus \\[0.2cm]
Institut f\"ur Mathematik \\
Universit\"at Rostock 
}
\maketitle

\setlength{\parindent}{0pt}
\begin{abstract}
\noindent  We consider  nonparametric statistical inference for L\'evy processes sampled irregularly, at   low frequency.  The estimation of the jump dynamics as well as the estimation of the distributional density are investigated.

\noindent Non-asymptotic risk bounds are derived and the corresponding rates of convergence are discussed  under global as well as local regularity assumptions. Moreover, minimax optimality is proved for the estimator of the jump measure. Some numerical examples are given to illustrate the practical performance of the estimation procedure.
\end{abstract}

\section{Introduction} 
Nonparametric statistical inference for stochastic processes with jumps  has  a long history, dating back as far as to the work by  \cite{Rubin_Tucker}  or \citet{Basawa_Brockwell}. 

In the past decade, jump processes have become increasingly popular among practitioners, especially  in the field of financial applications,  and  the interest  in the topic  has constantly grown. In particular,  a vast amount of literature has been published on  the estimation of the characteristics of L\'evy processes. 

In estimation problems for L\'evy processes, one can essentially distinguish between two different types of observation schemes:  In a high frequency framework, it is assumed that the maximal distance between the observation times tends to zero as the number $n$ of observations increases to infinity, whereas in a low frequency regime,  it is not assumed that the observation distances are asymptotically small.

So far, when low frequent observations of the underlying L\'evy process are considered, most publications have focused on the case where the observations are homogeneous, which means that the distance $\Delta$ between any two observation times $t_j$ and $t_{j+1}$ is fixed an does not vary with $j$, see for example, \citet{Neumann_Reiss}, \cite{Comte_Genon_1} or \cite{Kappus_SPA}.  Contrarily to this, when a high frequent sampling scheme is being investigated, the homogeneity assumption can easily be disposed of, see, among many others, the work by \cite{Figueroa_2} or \cite{Comte_Genon_2,Comte_Genon_3}. 

In the present work, we focus on  nonparametric statistical inference for L\'evy processes  when  the sampling scheme is low frequent and irregular.  This means that the (deterministic)  distances $\Delta_j$  between any two observation times $t_j$ and $t_{j+1}$  may vary in $j$, from very small values close to zero to some possibly very large value ~$\Dmax$.

 \cite{Belomestny_tcLevy} and, most recently,  \citet{Belomestny_Trabs} have investigated the estimation of the characteristics of a L\'evy process $(X_t)$, when homogeneous and low frequent observation of a time changed process $Y_t = X_ {\mathcal{T}(t) }$ are available. In this framework $\mathcal{T}$ is understood to be a random time change, independent of $(X_t)$. 

It is important to point out that the deterministic and irregular sampling scheme discussed in the present  work is not embedded, as as special case, in the time changed model which has been investigated in 
\cite{Belomestny_tcLevy}.  In that paper, the estimator is constructed for a stationary time-change process and under the standing assumption, referred to as (ATI), that $\E[\mathcal{T}(t) ] = t$.  However, when a deterministic time change is considered, this would readily imply   that $ \mathcal{T}(t) = t$  and the problem would hence degenerate to a standard estimation problem with homogeneous observations. 

The statistical framework considered in the present publication, with arbitrary irregular sampling,  is hence new in the literature on nonparametric estimation for L\'evy processes. We focus on the following problems: Firstly, the nonparametric estimation of the jump measure is being discussed under some additional assumptions on the process. Secondly, we discuss the estimation of the distributional density of $X_1$ under very general a priori assumptions on the process. This second problem has not yet received much attention in the literature on L\'evy processes and is, indeed, quite standard when high frequent or heterogeneous observations are available. However,  under low frequent and irregular sampling, the estimation of distributional densities is not straightforward. 

This paper is organized as follow: In Section \ref{Abschnitt_Modell}, the statistical framework is made precise and some notation is introduced. In Section \ref{Abschnitt_Sprungschaetzung},  the estimation of the jump measure is investigated. An oracle type estimator for the L\'evy density is introduced, which turns out to depend on the appropriate choice of certain weight functions.  Non-asymptotic bounds are derived for the oracle estimator and corresponding rates of convergence are derived and shown to be optimal in a minimax sense.  The density estimation is then addressed in Section  \ref{Abschnitt_Dichteschaetzung}.  In Section \ref{Abschnitt_Numerik},  an algorithm for the fully data driven choice of the weights as well as for the choice of the cutoff parameter is introduced. All proofs are postponed to Section \ref{Abschnitt_Beweise}. 
\section{Statistical model, assumptions and notation}\label{Abschnitt_Modell}
A  one dimensional L\'evy process $X=\{X_t:t\geq 0\}$ is observed at deterministic time points $0=t_0 < \cdots < t_n=:T$. Throughout the rest of this work, we assume that there exists a positive constant $ \Delta_{\max}$ such that $  \forall  j\in \N, \  t_j- t_{j-1}\leq \Delta_{\max}$ .  Apart from this upper bound, no additional assumptions on the observation times are imposed so the sampling is irregular and fully general. 

The goal of this paper is twofold. Firstly, we focus on the estimation of the jump dynamics under the following additional assumptions. 
\begin{itemize}
\item[(A1)]$X$ has finite variation on compact sets.
\item[(A2)]$X$ has no drift component.
\item[(A3)]For one and hence for any $t>0$, $\E[X_t^2]<\infty$. 
\item[(A4)] The   \Levy\ measure $\nu$ has a Lebesgue density 
$\eta$ which is continuous on $\R \setminus\{0\}$. 
\end{itemize}

Under  (A1)-(A4), the characteristic function of $X_{\Delta}$ is known to admit  the following representation. 
\begin{equation}
\label{Darstellung_cf}
\phi_{\Delta}(u):=\E\left[e^{iuX_{\Delta}} \right]=e^{\Delta \Psi(u)},
\end{equation}
with characteristic exponent 
\begin{equation}
\label{Darstellung_ce}
\Psi(u)=\int \left( e^{iux} - 1 \right)  \eta(x)\d x= \int \frac{e^{iux} - 1    }{ x } x \eta(x) \d x.
\end{equation}
For the proof, see e.g. Theorem 8.1 in \cite{Sato}.  The process is then entirely described by the jump measure and hence by the function $g(x):=x\eta(x)$.  We discuss the estimation of $g$, with $\lk^2(\Omega)-$loss on some interval $\Omega:=[\omega_1,\omega_2]$, $-\infty\leq \omega_1< \omega_2\leq \infty$.  

The above assumptions are met  for many prototypical L\'evy processes  such as  compound Poisson processes,  gamma processes or  tempered stable processes without drift and with  index of stability $\alpha \in (0,1)$.  Notice that $g\in \lk^2(\R)$ may fail to hold true,  but is  satisfied if $\eta(x) =O(|x|^{-\alpha}), \ |x|\to 0$ for some $\alpha<\frac{3}{2}$. This is met, for example, for compound Poisson processes, gamma-processes and tempered stable processes with $\alpha\in (0,1/2)$.  If 
$g\not\in \lk^2(\R)$, estimating $g$ with $\lk^2(\Omega)$-loss does still make sense for  compact sets $\Omega$   bounded away from the origin.

It is worth mentioning that the assumptions (A1) and (A2) can be omitted  and a fully general treatment of the problem is possible, but at the cost  of additional technical complications, see, for example,  \cite{Neumann_Reiss}
for the homogeneous case. 

Secondly, we investigate the estimation of the distributional density $f$ of $X_1$ with $\lk^2(\R)$-loss, in case that  a square integrable Lebesgue density exists.  It is well known that $X_1$ possesses a density $f\in \lk^2(\R)$   if the process has a high enough activity of small jumps or a non-zero Gaussian component. This follows, for example, from arguments given in \cite{Orey}.  For example, the Lebesgue density exists and is square integrable  if  the process is of pure jump type and  there exist constants $C>0$ and $\beta >\frac{1}{2}$ such that \mbox{ $\eta(x)+\eta (-x) > \beta |x|^{-1}\  \forall x\in[0,C]$}.

Typical examples are  gamma processes with scale parameter $\beta >\frac{1}{2}$,  stable or tempered stable processes as well as the Brownian motion. Prototypic counterexamples are compound Poisson processes or gamma processes with $\beta \leq \frac{1}{2}$.  


We conclude this section by introducing some notation which will be used throughout the rest of the text.

For a function $f\in \lk^1(\R)\cup\lk^2(\R)$,  $\Fourier f$ is understood to be the Fourier transform.  By $\Delta_j$, we denote the distance $t_{j}- t_{j-1}$ between the observation times.  $Z_j= X_{j}- X_{j-1}$ is the corresponding increment of the process and 
\[
\phi_j(u):= \ebig{ e^{iu Z_j}  }   = \ebig{ e^{iu X_{{\Delta_j}} }  }
\]
its characteristic function. Given a kernel $\kf$ and bandwidth $h>0$, we write $\kf_h(x) := 1/h \kf(x/h)$.  
Moreover, $\DMax:=\max\{\Dmax, 1\}$.  Let $X_+$ and $X_{-}$ be independent L\'evy processes with L\'evy measures $\nu_+(\d x)= \nu|_{(0,\infty)}(\d x)$  and $\nu_{-}(\d x):= \nu|_{(-\infty,0)}(\d (\m x))$. For $m\in \N$, 
\mbox{ $C_m^{\pm}:= \E[X_+^m] +\E[X_{-}^{m}]$}.  Finally, given a continuous function $f$ and $u\in \R$ 
\[
\| f \|_{\lk^k, u}^k := \int_{-u}^u  |f(z)|^k \d z   \quad \text{and} \quad \| f\|_{\infty, u}:= \sup_{z\in [-u,u] }|f(z)|. 
\]
\section{Estimation of the jump dynamics} \label{Abschnitt_Sprungschaetzung}
\subsection{Estimation procedure and non-asymptotic risk bounds}
\label{Abschnitt_Schaetzer}

It follows from formula \eqref{Darstellung_cf} and formula \eqref{Darstellung_ce}  that  the Fourier transform of $g$ can be recovered by differentiating the characteristic exponent,
\begin{equation}
\label{char_exp_abl}
\Psi'(u) =
\int \frac{\d}{\d u} e^{iux } \nu(\d x) = i \int e^{iux}  x\eta(x) \d x \d x= i \Fourier g(u).
\end{equation}
For (possibly complex) weight functions  $w_{1},\cdots, w_{n}$ to be chosen appropriately , we define 
\begin{equation}
\label{p_Definition}
p(u) := \sum_{j=1}^n\limits  w_j (u) \phi_{\Delta_j} '(u)
=  \sum_{j=1}^n\limits  w_j (u) \Delta_j\Psi'(u)  \phi_{{\Delta_j}}(u)
\quad \text{
and }
\quad 
q(u) := \sum_{j=1}^n\limits w_j(u) \Delta_j \phi_{\Delta_j}(u). 
\end{equation}
The derivative of the characteristic exponent then equals the ratio $p/q$ and formula  \eqref{char_exp_abl}  permits to recover the Fourier transform of $g$ as follows,  
\begin{align}
 \Fourier g(u)= \Psi'(u) /i =\frac{p(u)}{i q(u) }. 
\end{align}

Let us introduce the empirically accessible counterparts of $p$ and $q$, 
\begin{align}
\hat{p}(u):=\sum_{j=1}^{n}    w_{j}(u)  \hat{\phi}'_j(u) :=  \sum_{j=1}^{n} w_{j}(u) i Z_j e^{iu Z_j} 
\end{align}
and
\begin{align}
 \hat{q}(u) := \sum_{j=1}^{n} \Delta_j w_{j}(u) \hat{\phi}_j(u):= \sum_{j=1}^{n} \Delta_j w_{j}(u) 
e^{iu Z_j} .
\end{align}
Taking expectation, we find that these quantities are unbiased estimators of $p$ and $q$. This suggests to use  $\hat{p}/\hat{q}$ as an estimator of $\Psi'$.

However, the estimation of the L\'evy density from low frequent observations is a prototypical statistical inverse problem and the rates of convergence are  governed by the smoothness of the underlying density as well as the decay behavior of the function $q$ in the denominator.  Too small values of the denominator will typically lead  to a highly irregular behavior of the estimator and hence a large variance.  Inspired by  \cite{Neumann}, we introduce a regularized version of the inverse of $\hat{q}$. Once $\hat{q}$ is below some threshold which is of the order of the standard deviation, a reasonable estimate of $1/q$ is no longer possible so the estimator is set to zero. Moreover, in order to ensure that the estimator is bounded from below, we introduce an additional constant threshold value.  For some threshold parameter $\kappa>0$ to be chosen, 
\begin{align}
\frac{1}{\tilde{q}(u)}:=\frac{1\left(\{|\hat{q}(u)|\geq \max\{ \sigma(u),\kappa\} \}\right)}{\hat{q}(u)}, \quad \text{with}  \quad \sigma(u):= \sqrt{  \sum_{j=1}^{n}   \Delta_j^2 |w_j(u)|^2\,  }. 
\end{align}
The corresponding regularized estimator of $\Psi'$ is    $\hat{\Psi'}(u):=\hat{p}/\tilde{q}$.  Finally,  let   $\kf$ be a kernel  function having an integrable Fourier transform.   For $h>0$,  the corresponding kernel estimator of $g$ is 
\begin{equation}
\hat{g}_{h}(x):= \Infourier \left( \Fourier \kf_h  \frac{\hat{p} }{ i \tilde{q} } \right)(x)
= \frac{1}{2 \pi } \int e^{\m iux }  \Fourier \kf(hu) \frac{\hat{p}(u) }{i \tilde{q}(u) }  \d u .
\end{equation}

There remains to specify the $w_j$. It  is intuitive that the optimal choice  of $w_j$  will depend on  $\Delta_j$ as well as on $\phi_j$.  For small values of $\phi_j$, the noise dominates so the quality of the estimator gets worse, which should lead to choosing a relatively small weight. Moreover, the smaller  $\Delta_j$ is, the more information on the jump dynamics  is contained in the observation, which motivates to give a high weight to $\hat{\phi}_j$. These considerations lead to specifying the ideal (oracle) weights 
$w_{j}^*(u):=\overline{ \phi_{{\Delta_j}} (u) }$.   Notice that this statistical framework has 
 strong  structural similarities  to  a deconvolution problem with heteroscedastic errors, see \cite{Delaigle_Meister}.  

In what follows, $\hat{g}_h$  is understood to be the oracle estimator corresponding to the ideal weights 
$w_j= w_j^{*},\, j=1, \cdots, n$.   It is clear, however, that the $w_j^*$ are not feasible to actually compute,  since they depend on the (unknown) characteristic function.  In the sequel, we derive risk bounds and optimality properties for the oracle estimator $\hat{g}_h$.  A procedure for the fully data driven choice of the weights is then  proposed in Section \ref{Abschnitt_Numerik}. 

\newcommand{\gest}{\hat{g}_{h,n}}
\renewcommand{\E}[1]{\mathbb{E}\left[#1\right] }
\begin{theorem}
\label{Hauptsatz_Risiko} Assume that $\E{ |X_1|^4 }<\infty$.  
Assume, moreover, that $g|_\Omega  \in \lk^2(\Omega)$.  Then there exists some $C: \R_+^2\to \R_+$  which is monotonously increasing with respect to both components such that 
\begin{align}
\E{\|\hat{g}_{h,n} - g     \|^2_{{\lk^2}(\Omega)}  }\leq & 2\| g- \kf_{h}\ast g\|_{{\lk^2(\Omega) }}^2 +C(  \DMax, C^{\pm}_{4} )  \int \frac{ |\Fourier \kf(h u)|^2 }{q(u) }\   \textrm{d}  u\\
= &2\| g- \kf_{h}\ast g\|_{{\lk^2(\Omega) }}^2 +C(  \DMax, C^{\pm}_{4} )  \int \frac{ |\Fourier \kf(h u)|^2}{\sum_{j=1}^{n} \Delta_j |\phi_{{\Delta_j}}(u)|^2  }\   \textrm{d}  u .
\end{align}
\end{theorem}
\textbf{Discussion.}  It is interesting to note that the upper risk bound  in the preceding theorem confirms the analogy of the nonparametric estimation of $g$ with a deconvolution problem with heteroscedastic errors, see \cite{Delaigle_Meister}. 

We concentrate, in this work, on a deterministic sampling scheme. However, the construction of the estimator and the upper bounds may be generalized to the case where the process is sampled at random times, provided that the random sampling does not depend on the underlying process~$X$. In the proof of the rate results, one will then have to consider conditional expectations with respect to $\mathcal{T}$. 
\subsection{Rates of convergence}
\subsubsection{Minimax upper  bounds} 
In this section, we study the asymptotic properties of $\hat{g}_h$ which can be derived from the upper risk bound formulated in Theorem \ref{Hauptsatz_Risiko}. 
\subsubsection*{Global regularity}
We start by considering the estimation of $g$ with $\lk^2$-loss on the whole real axis and the resulting rates of convergence over certain nonparametric classes. It has been pointed out that the estimation of $g$ resembles the estimation of a distributional density from observations with additional heteroscedastic errors.  The rates of convergence will thus depend on the $\Delta_j$ and on the decay of the  $|\phi_j|=|\phi|^{\Delta_j} $, as well as on the smoothness of $g$. 

However, one needs to beware of the fact that (unlike in a standard deconvolution framework) the smoothness of $g$ and the decay of $\phi$ are not independent of each other. 

It is easily seen that the polynomial or exponential decay conditions  
\[
 |\phi(u)|\sim |u|^{-\beta}, \   u\to \infty   \quad \text{or}   \quad |\phi(u)|\sim \exp ( -c|u|^{\alpha}) 
\]
imply a logarithmic or polynomial growth of the characteristic exponent which gives, in turn, 
\[
\limsup_{u\to \infty} \frac{|\Fourier g(u)|}{|u|^{-1} }  =\limsup_{u\to \infty} \frac{|\Psi'(u) |}{|u|^{-1} }
>0 \quad \text{or}   \quad \limsup_{u\to \infty} \frac{|\Fourier g(u)|}{|u|^{\alpha-1} }>0. 
\]
The faster the characteristic function $\phi$ of $X_1$ decays, the slower will hence be the decay of the Fourier transform of $g$.  For processes of compound Poisson type, the absolute value of  $\phi$ is bounded from below and $g$ may be infinitely differentiable. \\[0.2cm]
Let us introduce the following nonparametric classes of functions and of corresponding L\'evy processes: For $\beta >0$ and $\alpha\in (0,1/2)$, let $\mathcal{G}_{\tiny{\text{pol}} }(\beta,C_\phi,c_\phi, C_g,C)  $ or  $\mathcal{G}_{\tiny{\text{exp}} }(\alpha,C_\phi,C_g,C)$ be the classes of functions $g(x)= x\eta (x)$ such that for the corresponding L\'evy process, (A1)-(A4) are met,  $C_4^{\pm}<C$,
\[
\forall u\in \R: \  |\phi_{{X_1}}(u)| \geq  (1+C_\phi|u|^{2})^{-\frac{\beta}{2} }
\quad \text{or}   \quad \forall u\in \R: \  |\phi_{{X_1}}(u)| \geq  C_\phi \exp(-c_\phi |u|^{\alpha} ),
\] 
and, in addition, 
\[
\forall u\in \R: |\Fourier g(u)| \leq C_g |u|^{-1}\quad \text{or}   \quad \forall u\in \R: \
|\Fourier g(u)|  \leq C_g  |u|^{\alpha -1}.
\]
Moreover, let $\mathcal{G}_{\tiny \text{cp} }(C_\phi, a, \rho, C_g,c_g,C)$ be the class of functions $g$  such that  $X$ is a compound Poisson process, $C_4^{\pm}<C$,
\begin{align}
\forall u\in \R: \  |\phi_{{X_1}}(u)| \geq  C_\phi
\end{align}
and 
\begin{align}
\forall u \in \R: |\Fourier g(u)|  \leq C_g |u|^{-a} \exp(-c_g |u|^{\rho}  ) . 
\end{align}
The following result which describes the rates of convergence with respect to the prescribed smoothness classes introduced above is a direct consequence of Theorem \ref{Hauptsatz_Risiko}. 
\begin{proposition}  Let $\kf$ be the sinc-kernel. This is equivalent to stating that  $\Fourier \kf(u)= \mathds{1}_{[-1,1]}(u)$.  
\begin{itemize}
\item[(i)] Let $h^*$ be implicitly defined, as the solution of the minimization equation 
\begin{equation}  \label{Bandweite_glo_pol}
 \sum_{j=1}^{n}  \Delta_j  h^{2\Delta_j \beta+2} = 1.
\end{equation}
Then 
\begin{align}
\sup_{ g\in \mathcal{G}_{\tiny{\text{pol}} }(\beta,C_\phi,c_\phi, C_g,C)    } 
  \mathbb{E}_g\big[\| g-\hat{g}_{{h^*}}  \|_{\lk^2(\R)}^2 \big] = O(h^*  ). 
\end{align}
\item[(ii)] Let $h^*$ be the solution of 
\[
 \sum_{j=1}^{n}  \Delta_j  e^{-2\Delta_j c_\phi(1/h)^{\alpha}}  h^{2(\alpha -1)}= 1.
\]
Then 
\begin{align}
\sup_{ g\in \mathcal{G}_{\tiny{\text{exp}} }(\alpha,C_\phi,C_g,C)   } 
  \mathbb{E}_g\big[\| g-\hat{g}_{{h^*}}  \|_{\lk^2(\R)}^2 \big] = O\big( (h^*)^{1-2\alpha}  \big). 
\end{align}
\item[(iii)] Let $h^*$ be the solution of 
\[
e^{2 c_g (1/h)^\rho } h^{-2a} = \sum_{j=1}^{n} \Delta_j C_\phi^{\Delta_j}. 
\]
\begin{itemize}
\item[a)] Assume that $\rho>0$ holds and   $a=(1-\rho)/2$.  Then
\begin{align}
\sup_{ g\in\mathcal{G}_{\tiny \text{cp} }(C_\phi, a, \rho, C_g,c_g,C) } 
  \mathbb{E}_g\big[\| g-\hat{g}_{{h^*}}  \|_{\lk^2(\R)}^2 \big] = O\big(\exp(- 2c_g (1/h^*)^{\rho} ) \big).
\end{align}
\item[b)] Assume that $\rho=0$ and $a>0$. Then 
\begin{align}
\sup_{ g\in\mathcal{G}_{\tiny \text{cp} }(C_\phi, a, \rho, C_g,c_g,C) } 
  \mathbb{E}_g\big[\| g-\hat{g}_{{h^*}}  \|_{\lk^2(\R)}^2 \big] = O\big( (h^*)^{2a-1}  \big). 
\end{align}
\end{itemize}
\end{itemize}
\end{proposition}
The convergence rates summarized above, with $h^*$ defined implicitly, are not particularly intuitive. For this reason, we give some examples to better understand the underlying structure. \\[0.2cm]
\textbf{Examples.} 
\begin{itemize}
\item[(i)] Consider the special case of homogeneous observations, $\Delta_j=\Delta,\  j=1,...,n$. Then, for the function class $\mathcal{G}_{\text \tiny pol}$, the implicit definition of the bandwidth  implies that  $h^*\asymp T^{-\frac{1}{2+2\Delta \beta} }$.  This leads  to the convergence rate 
\begin{align}
\sup_{  g\in \mathcal{G}_{\tiny{\text{pol}} }(\beta,C_\phi,c_\phi, C_g,C)   }\mathbb{E}_g\left[ \left\|g-\hat{g}_{h^*}\right\|_{\lk^2(\R)}^2 \right] = O\left( T^{-\frac{ 1}{2+2\Delta \beta } }  \right),
\end{align} 
which is optimal for estimating g with $\lk^2(\R)$-loss from homogeneous observations.  For the class $\mathcal{G}_{\text{ \tiny exp} }$, we find that $h^*\asymp (\log T/ (2 c_\phi \Delta ))^{-\frac{1}{\alpha}}$. The corresponding convergence rate is logarithmic, $O( \Delta^{1-2\alpha} (\log T)^{-\frac{1-2\alpha}{\alpha} })$.  

(Notice that in the homogeneous case, the weights coincide for all $j=1, ...,n$ and hence cancel in the construction of the estimator. ) 
\item[(ii)]  For non-homogeneous observations, the rates of convergence are intimately connected to the maximal distance between the observation times.  Consider the class $\mathcal{G}_{\text{\tiny pol}} $.  Part (i) of Proposition \ref{Hauptsatz_Raten} then implies that  $h^*  < T^{-\frac{1}{2+2\beta\Dmax}  }$ and consequently, 
\begin{align}
\sup_{  g\in \mathcal{G}_{\tiny{\text{pol}} }(\beta,C_\phi,c_\phi, C_g,C)   }\mathbb{E}_g\left[ \left\|g-\hat{g}_{h^*}\right\|_{\lk^2(\R)}^2 \right] = O\left( T^{-\frac{ 1}{2+2\beta\Dmax  } }  \right).
\end{align}
In particular, for $\Dmax$ small, we approach the rate of convergence $T^{-\frac{1}{2}}$ which is optimal for estimating a density $g$ with $\lk^2(\R)$-loss when $|\Fourier g(u)|\sim |u|^{-1}$.  The estimation of $g$ can hence be understood in analogy with density deconvolution whenever the observations are low frequent, and in analogy with standard density estimation when $\Dmax$ is small. 

In the compound Poisson case, for $\rho>0$,   we find that  $h^*\asymp (\log T)^{-\frac{1}{\rho}}$.  The rate of convergence  is $(\log T)^{\frac{1}{\rho}} T^{-1}$ and hence coincides with the optimal convergence rate for standard density estimation problems with supersmooth densities. In this particular case, $\Dmax$ does not affect the rate of convergence but only appears as a constant factor. 
\end{itemize}
\subsubsection*{Local regularity} 
It has been pointed out in the preceding section that the global regularity of $g$ is linked to the decay behavior of $|\phi|$, which influences the rates of convergence to be obtained. In particular, an exponential decay of $|\phi|$ will always lead to slow logarithmic rates of convergence. 

However, one may as well be interested in the estimation of $g$ on some compact set $\Omega$ bounded away from the origin.  The L\'evy density and hence the function $g$ may then be arbitrarily regular and even infinitely differentiable on $\Omega$. In what follows, we investigate rates of convergence over nonparametric classes of locally smooth functions. 

Recall that for an open interval $D$, the \emph{H\"older class} $ \mathcal{H}_D(a,L, R)  $ consists of those functions $f$ defined on $D$, whose absolute value is bounded above by $R$, which are  $\lfloor a \rfloor$-times continuously differentiable and  for which 
\[
\sup_{ x\not =y \in D}   \frac{|f^{( \lfloor a \rfloor)}(x) - f^{( \lfloor a \rfloor)}(x)     | }{|x-y|^{\alpha -  \lfloor a \rfloor   }    } \leq L.
\]
holds.  In the present case, $\lfloor a \rfloor$ denotes the largest integer which is smaller than (but not equal to)~$\alpha$. 

Let us introduce nonparametric classes of locally H\"older regular functions and corresponding L\'evy processes. In the sequel,   $\mathcal{G}(a,D,L,R,C_1,\beta, C_2, C_3)$  denotes the class of functions $g$ for which the following holds: 
\begin{itemize}
\item[(i)] There exists a L\'evy process $X$  for which (A1)-(A4) holds, with L\'evy density $\eta(x) =g(x)/x$. 
\item[(ii)] $g|_{D}$ can be extended to a function  $\tilde{g}\in \mathcal{H}_{\R}(a,L,R)$  and 
$\|g\|_{\lk^1(\R) }+ \|\tilde{g}\|_{\lk^1(\R)  } <C_1$.    
\item[(iii)]  $\forall u\in \R_+:  |\phi_{{X_1}}(u)| \geq(1+ C_2 |u|^2)^{-\frac{\beta}{2}}$.
\item[(iv)]  $X$ has a finite fourth moment and $ C_4^{\pm} \leq C_3$.
\end{itemize}
The following lemma gives a bound on the bias term for locally H\"older regular functions.

\begin{lemma}
\label{Lemma_Bias} Let $\Omega$ be a compact interval, bounded away from the origin. Assume that there exists an open interval $D\supseteq \Omega$ such that $g|_D\in \mathcal{H}_D(a,L,R)$  and $g|_D$ can be extended to a  function $\tilde{g} \in \lk^1(\R)\cap \mathcal{H}_{\R}(a,L,2 R)$.  Let the kernel be chosen such that $\kf\in \lk^2(\R)$, $\kf$ has the order  $a>0$ and moreover, 
for some constant $C_{\kf}$, 
\begin{equation}
\label{Abfall_Kern}
\forall x\in \R:  \left|\kf(x) \right| \leq C_{\kf} |x|^{-a-1}.
\end{equation}
Then we can estimate for some positive constant $C_b$ depending on  the choice of $\kf$, on $L, R$ and on $\Omega$ and $D$
\begin{eqnarray}
\|g- \kf_h\ast g \|_{\lk^2(\Omega) }^2 \leq  C_b (1+\|g\|_{\lk^1(\R)}  + \|\tilde{g}\|_{\lk^1(\R)} ) h^{2a}.
\end{eqnarray}
\end{lemma}

\textbf{Comment.} It is important to realize that the condition \eqref{Abfall_Kern} on the kernel is crucial.  The sinc-kernel is thus not a reasonable choice in the present situation, for estimating $g$ on a compact set bounded away from the origin. This is a consequence of the fact that, no matter how $g$ may be locally, it will have discontinuity at zero when the jump activity is infinite, which implies that the function is globally non-smooth. \\[0.2cm]
The following proposition is a direct consequence of the bound on the bias given above, combined with Theorem \ref{Hauptsatz_Risiko}. 
\begin{proposition}
\label{Hauptsatz_Raten}
 Let $\kf$ be such that the assumptions summarized in Lemma \ref{Lemma_Bias} are met.   Moreover, let $\kf$ be supported on 
$[-1,1]$.  Let $h^*$ be implicitly defined as the solution of
\begin{equation} \label{Bandweite_lo_pol}
1= \sum_{j=1}^{n} \Delta_j  h^{ 2 \beta \Delta_j+2a+1}. 
\end{equation} 
Then 
\begin{equation}
\sup_{g\in  \mathcal{G}(a,D,L,R,C_1,\beta, C_2, C_3)}\mathbb{E}_g\left[ \left\|g-\hat{g}_{h^*}\right\|_{\Omega}^2 \right] = O\left( {{h^*}}^{2a}  \right). 
\end{equation}
\end{proposition}
\textbf{Example.} Again, we investigate how $\Dmax$ influences the rate of convergence to be obtained. 
We find that $h^*\lesssim h'$, with
\begin{align}
h'=  T^{-\frac{1}{2a+2\Dmax\beta+1} }.
\end{align} 
The resulting rate of convergence is faster than or equal to $O(T^{-\frac{2a}{2a+2\Dmax\beta+1}})$. This implies that rate approaches the rate of convergence $T^{-\frac{2a}{2a+1} }$ which is known the be optimal for estimating $g$ (locally) with $\lk^2$-loss  when continuous time observations of the process are available.  \\[0.2cm]
\textbf{Comment.} In the same spirit, we may consider the case where $|\phi(u)|$ decays exponentially,
$|\phi(u)| \sim e^{-c|u|^\gamma}$.  It can then be shown that, with optimal choice of $h^*$, one attains convergence of order   $(\log T)^{-\frac{2a}{\gamma} }$. 

\subsection{Lower bounds}
In the sequel, we show that the rates of convergence presented in the preceding paragraphs are  minimax optimal.  However, to avoid some technical difficulties,  we content ourselves with considering the case of polynomially decaying characteristic functions. 
\begin{theorem}\label{Satz_untere_Schranken}
\begin{itemize}
\item[(i)]
Let $h^*$ be defined according to  \eqref{Bandweite_glo_pol}.  Then there exists a positive constant $c$ such that 
\begin{equation}
\liminf_{n\to \infty}\inf_{\check{g}} \sup_{ g\in \mathcal{G}_{\tiny{\text{pol}} }(\beta,C_\phi,c_\phi, C_g,C)   }   \mathbb{E}_{g} \left[\| g- \check{g} \|_{\lk^2(\R)}^2  \right]  (h^*)^{-1} \geq c.
\end{equation}
The infimum is taken over the collection of estimators $\tilde{g}_n$ based on the observations
$X_{{t_1}}, \cdots, X_{{t_n}}$.
\item[(ii)] Let $h^*$ be defined according to  \eqref{Bandweite_lo_pol}.  Then there exists a positive constant $c$ such that 
\begin{equation}
\liminf_{n\to \infty}\inf_{\check{g}} \sup_{g\in  \mathcal{G}(a,D,L,R,C_1,\beta, C_2, C_3)}\mathbb{E}_g\left[ \left\|g-\check{g}\right\|_{\lk^2(\Omega) }^2 \right](h^*)^{-2a} \geq c. 
\end{equation}
\end{itemize}  
\end{theorem}
\section{Estimating the distributional density}\label{Abschnitt_Dichteschaetzung}

Let us have a look at the situation where one is interested in estimating the density  $f$ of $X_1$ rather than the underlying L\'evy measure.   We can now drop the technical assumptions (A1)-(A4) on the process, thus allowing a non-zero Gaussian part, a high activity of small jumps and the existence of a drift.  In the sequel, it is only assumed that the density $f$ of $X_1$ exists and is square integrable on the whole real axes, thus excluding processes of compound Poisson type.  

 Without specifying any particular assumptions on the L\'evy measure, drift or Gaussian part, the estimation procedure proposed in Section 3  can still  be used to build an estimator of  the derivative of the characteristic exponent,
\begin{align}
\hat{\Psi'} (z) : = \frac{\hat{p} (z) }{\tilde{q}(z) }.
\end{align}
Since $\Psi(0)=0$ holds by definition of the characteristic exponent, a corresponding estimator of the characteristic exponent can be defined as follows. 
\[
\hat{\Psi}(u) : = \intu  \hat{\Psi}'(z) \d z =\intu \frac{\hat{p}(z)}{\tilde{q}(z)} \d z.
\] 
The characteristic function $\phi(u)=\exp(\Psi(u) )$ of $X_1$   can then be estimated by  $ \check{\phi}(u) :=e^{\hat{\Psi}(u)}$.

However, there is a priori no guarantee that $\check{\phi}$ is a characteristic function and the absolute value may be larger than one. For this reason, we introduce an additional threshold, thus defining the final estimator of $\phi$  
\begin{equation}
\hat{\phi}(u) := \frac{ \check{\phi}(u)  }{\max\{1,|\check{\phi}(u)|\}}.
\end{equation}
Given a kernel function $\kf$ and bandwidth $h>0$, the density $f$  is then estimated using kernel smoothing and Fourier inversion, 
\begin{equation}
\hat{f}_{h}(x):= \frac{1}{2\pi}  \int e^{-iux} \Fourier \kf(h u) \hat{\phi}(u) \d u.
\end{equation}
\begin{theorem} \label{Schranken_DS} Let $\kf$ be supported on $[-1,1]$. Assume that $\E{|X|^{4m} }<\infty$.

Then there exists some $C: \R^2  \to \R_+$ which is monotonously increasing in both components  such that 
\begin{align}
&  \ebig{ \| f  -\hat{f}_{h } \|_{\lk^2}^2    }   
\leq    2 \| f- \kf_h\ast f \|_{\lk^2}^2+ C(\DMax, C_{4}^{\pm}  )  \inth |\phi(u)|^2 C_{\Psi,1}(u)  \intu \frac{1}{|q(x)|}  \d x  \d u  \\
& +   \inth   \Big(C_{\Psi,1}(u)  \intu \frac{1}{|q(x)|}  \d x \Big)^m  \max\Big\{1, C_{\Psi,2} (u)      \intu \frac{1}{|q(x)|}  \d x \Big\}^m\d u,
\end{align}
with 
\[
C_{\Psi,1}(u):= ( \|\Psi'\|_{\lk^2,u}^2 +\|\Psi''\|_{\lk^1,u}  ) \vee 1  \quad   \text{and}
\quad C_{\Psi,2}(u): = \| \Psi'\|_{\infty,u} .
\]
 \end{theorem}
\subsubsection*{Rates of convergence}
For $\beta>\frac{1}{2}$, let  $\mathcal{F}_{\text{\tiny pol} }(\beta, m, C_1,C_2,C_3, C )$ be the class of  densities  corresponding to an infinitely divisible distribution for which the following holds.
\begin{itemize}
\item[(i)]  For the characteristic function $\phi$ of $f$, 
\begin{align}
\forall u \in \R:      ( 1+C_1|u|)^{-\beta  }  \leq |\phi(u)|  \leq ( 1+C_2|u|)^{-\beta  }.
\end{align}
\item[(ii)]  $\Psi'$ is square integrable and $\Psi''$ is integrable, with 
\[
\|\Psi'\|_{\lk^2(\R)}^2 +\|\Psi''\|_{\lk^1(\R)}  +\|\Psi'\|_{\infty}   < C_3
\] 
\item[(iii)] The random variable $X_1$ with density $f$ has finite moments up to order $4m$  and  $C_{4m}^{\pm}\leq C $.
\end{itemize}
This function class contains, for example, the densities of gamma- or bilateral gamma distributions. \\[0.2cm]
For $\alpha\in(0,2]$, let $\mathcal{F}_{\text{\tiny exp}  }(\alpha,m, C_1, C_2, C_3,C)$ be the class of infinitely divisible densities such that 
\begin{itemize}
\item[(i)] $ 
\forall u\in \R: C_1\exp( - c|u|^{\alpha}  ) \leq |\phi(u)| 
\leq  C_2 \exp( - c|u|^{\alpha}  )
$ 
\item[(ii)] $
  \|\Psi'\|_{\lk^2,u}^2 +\|\Psi'\|_{\lk^1,u}  +\|\Psi'\|_{\infty,u}  \leq \begin{cases}
C_3, &\text{for } \alpha <\frac{1}{2}  \\
C_3 (\ln |u| \vee 1) , & \text{for }  \alpha=\frac{1}{2}    \\
C_3 |u|^{2\alpha - 1}. &\text{for }  \alpha >\frac{1}{2} . 
\end{cases}  $
\item[(iii)]  The moments up to order $4m$ are finite  and $C_{4m}^{\pm} \leq C$. 
\end{itemize}
Typical examples are tempered stable laws with index of stability $\alpha$,  or (for $\alpha=2$) processes with non-zero Gaussian part.\\[0.2cm] 
In the sequel, we define, for the exponential or polynomial decay scenario, respectively,  
\[
|\phi_{\text{\tiny reg }} (u)| := (1+|u|)^{-\beta }  \quad \text{or } \quad |\phi_{\text{\tiny reg }} (u) |:=  \exp( -c |u|^{\alpha} ) . 
\]
Moreover, 
\[
q_{\text{\tiny reg}}(u) := \sum_{j=1}^{n} \Delta_j |\phi_{\text{\tiny reg }}(u) |.
\]
The convergence rates over nonparametric function classes summarized in the following Proposition are immediate consequences of Theorem \ref{Schranken_DS}, performing the compromise between the bias and variance term. 
\begin{proposition}\label{Hauptsatz_Raten_DS}  Let $\kf$ be the sinc-kernel, $\Fourier \kf =\mathds{1}_{[-1,1]}$.  
\begin{itemize}  
\item [(i)]Consider $\beta >\frac{1}{2}$ and $\alpha\in (0,1)$. Then, with $h^*$ implicitly defined implicitly, as the solution to 
\begin{align}\label{hdef_1}
|\phi_{\text{\tiny reg}}(1/h)|^2 = \Big(  \int_0^{1/h}\limits \frac{1}{q_{\text{\tiny reg}}   (x) }  \d x \Big)^k,
\end{align} 
 we derive that 
\begin{align}
\sup_{f\in \mathcal{F}_{\text\tiny{pol} }(\beta, m, C_1, C_2, C_3,C) } \ebig{\| f- \hat{f} \|_{\lk^2(\R)}^2 } = O \big( (h^*)^{2\beta -1}  \vee T^{-1}  \big)
\end{align}
and 
\begin{align}
\sup_{f\in \mathcal{F}_{\text\tiny{exp} }(\alpha, m, C_1, C_2, C_3,C) }   \ebig{\| f- \hat{f} \|_{\lk^2(\R)}^2 } = O \big(h^{\alpha -1}  \exp(-2c(1/h^*)^{\alpha}  ) \vee T^{-1}  \big)  . 
\end{align}
\item[(ii)] Consider $\alpha=\frac{1}{2}$ or $\alpha \in(\frac{1}{2},2]$.  Let $h^*$ be the solution to 
\begin{align} \label{hdef_2}
|\phi_{\text{\tiny reg}}(1/h)|^2 = \Big(\ln(1/h)   \int_0^{1/h} \limits \frac{1}{q_{\text{\tiny reg}}   (x) }  \d x \Big)^k
\quad \text{or}  \quad  |\phi_{\text{\tiny reg}}(1/h)|^2 = \Big( (1/h)^{2\alpha - 1}    \int_0^{1/h} \limits  \frac{1}{q_{\text{\tiny reg}}   (x) }  \d x \Big)^k.
\end{align}
Then it follows that 
\begin{align}
\sup_{f\in \mathcal{F}_{\text\tiny{exp} }(\alpha, m, C_1, C_2, C_3,C) }   \ebig{\| f- \hat{f} \|_{\lk^2(\R)}^2 } = O \big( (h^{\alpha -1}\vee 1)   \exp(-2c(1/h^*)^{\alpha}  ) \vee T^{-1}  \big)  . 
\end{align}
\end{itemize} 
\end{proposition}
Again, the implicit definition of the smoothing parameter is not particularly intuitive so we have a look at some examples to understand how the convergence rates are connected to the maximal distance between the observation times. 

\pagebreak[3]
\vspace*{0.2cm}
\textbf{Examples.}\begin{itemize}
\item[(i)]  In the polynomial decay scenario,  \eqref{hdef_1} implies that $h^*\gtrsim h'$, with 
\begin{align}
h'=  T^{-\frac{1}{2\Dmax  \beta +1+\frac{2\beta}{k} } }, 
\end{align}
and the corresponding convergence rate is faster than or equal to $O\left( T^{-\frac{2\beta -1}{2\Dmax\beta +1 +\frac{2\beta}{k} } }\vee T^{-1}\right)$. This implies, in particular, that for $\beta >1$, without any further restriction on the regularity of the observation times, the estimator attains the parametric rate of convergence if a finite $4k$-th moment exists and $\Dmax \leq 1- \frac{1}{\beta} - \frac{1}{k}$. 
\item[(ii)] In the exponential decay scenario with $\alpha<\frac{1}{2}$,  \eqref{hdef_1} implies that  $h^*\gtrsim h'$, with 
\[
h'  = \Big( \frac{  \ln T - \ln( \ln T)^{\frac{1- \alpha}{\alpha} }    }{2c (\Dmax+\frac{1}{k}) }     \Big)^{\frac{1}{\alpha} }.
\]
The corresponding rate of convergence is faster than or equal to  $(\ln T)^{\gamma}  T^{-\frac{1}{\Dmax +\frac{1}{k} }}\vee T^{-1} )$, with $\gamma= \frac{(1-\alpha)(\Dmax 1+1/k}{(\Dmax +1/k)\alpha}$.

It follows that the parametric rate is attained whenever the $4m$-th moment is finite and $\Dmax < 1- \frac{1}{k}$. 

For exponential decay with $\alpha\geq 1/2$, the arguments are the same, apart from an additional logarithmic loss. 
 \end{itemize}
\section{Numerical examples} \label{Abschnitt_Numerik}
\subsection{Practical calculation of the weights and  data driven bandwidth selection}
The weights assigned to each of the $\hat{\phi}_j$ are crucial for the construction of the estimator and the ideal weights $w_j^*=\overline{\phi_j}$ are not feasible to actually compute.  

On may consider the following selection algorithm for the weights. The interval $[0,\Dmax]$ is divided into $K_n$ disjoint intervals of equal length.  Set $\mathcal{J}_k=\{\ell: \Delta_\ell\in I_k\}$. Then, for $\Delta_j\in I_k$, a biased estimator of $w_j= \overline{\phi}_j$ can be constructed, setting 
\[
\hat{w}_j:= \frac{1}{|\mathcal{J}_k|}\sum_{\ell \in \mathcal{J}_k } e^{- iu Z_\ell }. 
\] 
It can be shown that this estimation procedure has good theoretical properties and  it preserves, up to some logarithmic loss, the upper risk bound and convergence rates which have been derived for the oracle estimator when the number $K_n$ of sub-intervals is logarithmic in $n$. However, in numerical examples, the practical performance of the estimator turns out to be unsatisfactory. For this reason, we cannot recommend this procedure in applications. 

Instead, we propose an iterative selection algorithm for the weights.  A preliminary estimator $\hat{\Psi}$ of $\Psi$ is calculated, using the initial weights $w_j\equiv 1$.  Improved  estimators of the weights are then given by 
$\hat{w}_j = \overline{\exp(\Delta_j\hat{\Psi} ) }$. These improved estimators are then applied to build a new estimator of $\Psi$. The procedure  is iterated until the  $\lk^2([-\sqrt{T},\sqrt{T}])$-distance between the empirical weights is sufficiently small.  

The algorithm can be summarized as follows. 

\begin{align}
&\forall j: \hat{w}_{j,0}\leftarrow 0 \\
&  \forall j: \hat{w}_{j,1}\leftarrow 1  \\ 
& m \leftarrow 1 \\
& \text{while} ( \exists  j:\,   \| \hat{w}_{j,m-1}- \hat{w}_{j,m} \|^2_{\lk^2([-\sqrt{T},\sqrt{T}]) }
> 1/T )  \\
& \{  \\
& \forall j:\, w_j\leftarrow \hat{w}_{j,m}   \\
&\hat{\Psi}(u)\leftarrow  \int_0^u  \hat{p}(z)/\tilde{q}(z) \d z    \\
&\forall j:  \hat{w}_{j,m+1}  \leftarrow   \overline{\exp(\Delta_j \hat{\Psi} ) }  \\
&m\leftarrow m+1\\
&\} \\
& return(w_1, \cdots, w_n) 
\end{align}

\subsection{Estimation of the jump measure}
We consider the particular case where the kernel is specified to be the sinc-kernel and $g$ is estimated with $\lk^2(\R)$-loss.    We consider the collection of cutoff parameters $\mathcal{M}=
\{m  \in \N:   m\leq \sqrt{T} \}$.  The data driven cutoff parameter $\hat{m}$ is calculated, using a leave-p-out cross validation strategy:  Given any subset $P=\{n_1,..,n_p\}\subseteq  \{1,\cdots, n\}$ with $p$ elements,  we calculate the estimator $( \hat{\Psi}' )^{(P)}$ of $\Psi'$, based on the data set 
$\{Z_j, \, j\in P \}$ as well as the estimator  $(\hat{\Psi}')^{(-P)}$ based on $\{Z_{j}, \ j\not \in P\}$.  In the definition of 
$(\hat{\Psi}')^{(P)}$ we omit the regularization of $\hat{q}$ in the denominator.   

By the oracle cutoff parameter $m^*$, we understand the minimizer of the $\lk^2$-loss, 
\[
m^*=\argmin_{m\in \mathcal{M} } \limits \int |\hat{g}_m(x) - \hat{g}(x) |^2  \d x.
\]
By the Plancherel formula, 
\begin{align}
m^* = \argmin_{m\in \mathcal{M} }  \limits  \int_{-m}^{m} |\hat{\Psi}'(u)|^2 \d u - 2 \Re \int_{-m}^{m}   \hat{\Psi}'(u) \overline{\Psi'(u)} \d u 
=:\ell( g, \hat{g}_m ).
\end{align}
Since  this  true loss function is not feasible to compute,  $\ell(g, \hat{g})$ is approximated by the empirically accessible quantity 
\[
\hat{\ell}( g, \hat{g}_{m} ): = \int_{-m}^{m} |\hat{\Psi}'(u)|^2 \d u - 2   {n\choose p}^{-1} \sum_{P\subseteq \{1,...,n\}  }\Re 
\int_{-m}^{m} ( \hat{\Psi}') ^{(P)}(u) \overline{ (\hat{\Psi}')^{(\m P)} (u) }   \d u,
\]
and the data driven cutoff parameter is defined to be 
\[
\hat{m}=  \inf_{m\in \mathcal{M} } \hat{\ell} (g, \hat{g}_m).  
\]

We work with the choice  $p=n/10$. The constant threshold value $\kappa$ is set equal to $1$. However, when considering all subsets of size $p$, the numerical complexity of the algorithm explodes. In praxis,  we do not consider all subsets, but content ourselves with splitting the data into $n/100$ disjoint blocks $B_1, \cdots, B_{n/100}$ of equal size and calculate the leave-p-out estimators based on all subsets of the form $P=B_{j}\cup B_{j+1} \cup ... \cup B_{j+9},\ { j=1, \cdots,  {
n/10 - 9}}$. 

In the definition of the estimator, the distances between the observations may be fairly arbitrary and no particular structural assumptions are imposed, despite of the fact that the distances are bounded from above.   We calculate, in a preliminary step, distances between the observations which are independent realizations of a $\mathcal{U}[u,t]$-distribution. We consider the cases where $t=6$ or $t=2$.

The cross validated estimator is compared to the "estimator" with oracle choice of the weights and oracle cutoff of the bandwidth.  Moreover, in order to  illustrate the relevance of the proper choice of the weights, we give, for comparison, the "estimator" where all weights are set equal to 1. Again, we use in this case the oracle choice of the bandwidth.  

The following examples are being considered. 
\begin{itemize}
\item[(i)] $X$ is a Gamma-process with parameters $3$ and $2$, so $X_1\sim \Gamma(3,2)$. 

\item[(ii)] $X$ is a compound Poissson process with intensity parameter $3$ and standard normally distributed jumps. 
\item[(iii)] $X$ is a symmetric bilateral Gamma process with parameters $2$ and $4$, so $X_1$ is the difference between two independent $\Gamma(2,4)$-distributed random variables. 
\end{itemize}

Based on $500$ independent realizations,  we calculate the empirical risk $r_{or}$ of the "estimator" with oracle weights and oracle bandwidth, the empirical risk $r_{ad}$ of the estimator with adaptive bandwidth selection and data driven choice of the weights, as well as the empirical risk $r_{eq}$ corresponding the the procedure with equal weights and oracle bandwidth.

In Table 1 below, we summarize the result, for observation distances which are generated using a $\mathcal{U}([0,6])$-distribution.  In Table 2, the results are given when the observation distances are generated, using a $\mathcal{U}([0,2])$-distribution.

\begin{table}[h!!]
\renewcommand{\arraystretch}{2}
\begin{tabular}{|p{1cm}||p{1.5cm}|p{1.5cm} | p{1.5cm}||p{1.5cm}|p{1.5cm}|p{1.5cm}|}
\hline 
\   & \multicolumn{3}{c||}{$\Gamma(3,2)$} &  \multicolumn{3}{c|}{Cpois($\mathcal{N}(0,1),3)$} \\
\hline 
n  & $r_{or}$  & $r_{ad }$  & $r_{eq}$   & $r_{or}$  & $ r_{ad}$   & $r_{eq}$  \\
\hline 
1000 & 0.5152 & 1.1669  &    1.3790   & 0.1663  &0.3330&  0.7350 \\
\hline 
5000    &  0.2936 & 0.5350 &  1.2297
  &  0.0495  &  0.0981  & 0.4845 \\
\hline 
10000    & 0.2097  & 0.3456&   1.1378
   &  0.0268 &  0.0529  & 0.3828
\\
\hline 
\hline
\   & \multicolumn{3}{c||}{$b\Gamma(2,4,2,4)$} \\
\cline{1-4 }
n   & $r_{or}$  & $r_{ad }$  & $r_{eq}$ \\
\cline{1-4} 
1000 & 0.2845777 & 0.4068848 & 0.6049949   \\
\cline{1-4}
5000&     0.1123  &  0.1533  &  0.5230    \\
\cline{1-4}
10000&    0.0687 &  0.0871&  0.4719 \\
\cline{1-4}
\end{tabular}
\caption{$\mathcal{U}([0,6])$-distributed observation distances   }
\end{table}
\begin{table}[h!]
\renewcommand{\arraystretch}{2}
\begin{tabular}{|p{1cm}||p{1.5cm}|p{1.5cm} | p{1.5cm}||p{1.5cm}|p{1.5cm}|p{1.5cm}|}
\hline 
\   & \multicolumn{3}{c||}{$\Gamma(3,2)$} &  \multicolumn{3}{c|}{Cpois($\mathcal{N}(0,1),3)$} \\
\hline 
n  & $r_{or}$  & $r_{ad }$  & $r_{eq}$   & $r_{or}$  & $ r_{ad}$   & $r_{eq}$  \\
\hline 
1000 &  0.3483  & 0.6648 &  0.8352   &   0.0784    & 0.1975  &  0.1937  \\
\hline 
5000 & 0.1632 &  0.2374 & 0.6637    & 0.0201  &  0.0388 &  0.0596     \\
\hline 
10000  & 0.1104 &   0.14625   &   0.5786&  0.0111 &  0.0189   & 0.0326   \\
\hline 
\hline
\   & \multicolumn{3}{c||}{$b\Gamma(2,4,2,4)$} \\
\cline{1-4 }
n   & $r_{or}$  & $r_{ad }$  & $r_{eq}$ \\
\cline{1-4} 
1000 &     0.1682   &   0.2484  &   0.3617  \\
\cline{1-4}
5000&       0.0460 &    0.0574   &  0.2314          \\
\cline{1-4}
10000&  0.0246 &  0.0269    & 0.1817   
    \\
\cline{1-4}
\end{tabular}
\caption{$\mathcal{U}([0,2])$-distributed observation distances   }
\end{table}

\subsubsection{Density estimation}
In the same way as in the preceding section, we define the 
estimators $(\hat{\Psi'})^{(P)}$ and  $(\hat{\Psi'})^{(-P)}$.   The corresponding estimators of the characteristic function are  
$\hat{\phi}^{(P)}(u):= \exp( \intu (\hat{\Psi'})^{(P)}(x) \d x ) $  and  $\hat{\phi}^{(-P)}(u):= \exp( \intu (\hat{\Psi'})^{(-P)}(x) \d x ) $.  The cutoff parameter is then defined to be 
\[
\hat{m}:= \inf_{m\in \mathcal{M} }  \int_{-m}^{m}  |\hat{\phi}(u)|^2 \d u  - 2 {n \choose p}^{-1}
\sum_{P\subseteq \{1,...,n\}   }  \int_{-m}^{m}  \Re \big( \hat{\phi}^{(P)}  (u)  \overline{  \hat{\phi}^{(-P)} (u) }\big)  \d u.
\]   
Again, we specify $p=n/10$ and consider, in praxis, only a part of the subsets of size $p$. 

The cross-validated estimator with empirical weights is once more compared to the estimator with oracle choice of the weights and of the cutoff parameter, as well as to the estimator where all weights are identical 1. 

We consider the following examples: 
\begin{itemize}
\item[(i)] Gamma-process with parameters $3$ and $2$.
\item[(ii)] Bilateral gamma-process with parameters $2$ and $4$. 
\item[(iii)] Brownian motion with variance 1 and drift 2. 
\end{itemize}
The results are summarized in Table 3 and Table 4 below. 
\begin{table}[h!]
\renewcommand{\arraystretch}{2}
\begin{tabular}{|p{1cm}||p{1.5cm}|p{1.5cm} | p{1.5cm}||p{1.5cm}|p{1.5cm}|p{1.5cm}|}
\hline 
\   & \multicolumn{3}{c||}{$\Gamma(3,2)$} &  \multicolumn{3}{c|}{$b\Gamma(2,4,2,4)$} \\
\hline 
n  & $r_{or}$  & $r_{ad }$  & $r_{eq}$   & $r_{or}$  & $ r_{ad}$   & $r_{eq}$  \\
\hline 
1000 & 0.00319   &  0.00418  &    0.06389  &   0.00369  &  0.00460   &    0.01777   \\ 
\hline 
5000  &  0.00095  &  0.00099  &  0.03510   &    0.00083
  &   0.00088 &    0.00760   \\
\hline 
10000   & 0.00060    &   0.00061 &  0.02889   &  0.00041      &0.00041    &0.00501       \\
\hline 
\hline
\   & \multicolumn{3}{c||}{$\mathcal{N}(2,1)$ } \\
\cline{1-4 }
n   & $r_{or}$  & $r_{ad }$  & $r_{eq}$ \\
\cline{1-4} 
1000 & 0.00140
& 0.00288 &    0.03632     \\
\cline{1-4}
5000 &0.00065  &  0.00065&           0.01972   \\
\cline{1-4}
10000& 0.00055&  0.00056 &   0.01547            \\

\cline{1-4}
\end{tabular}
\caption{Estimation of the distributional density, with $\mathcal{U}([0,6])$-distributed observation distances   }
\end{table}
\begin{table}[h!!!!]
\renewcommand{\arraystretch}{2}
\begin{tabular}{|p{1cm}||p{1.5cm}|p{1.5cm} | p{1.5cm}||p{1.5cm}|p{1.5cm}|p{1.5cm}|}
\hline 
\   & \multicolumn{3}{c||}{$\Gamma(3,2)$} &  \multicolumn{3}{c|}{$b\Gamma(2,4,2,4)$} \\
\hline 
n  & $r_{or}$  & $r_{ad }$  & $r_{eq}$   & $r_{or}$  & $ r_{ad}$   & $r_{eq}$  \\
\hline 
1000 & 0.00168   &  0.00178 & 0.00786 &    0.00185    &    0.00195      &   0.00305      \\ 
\hline 
5000  &0.00057  & 0.00057&  0.00205   &   0.00035   & 0.00036   &  0.00074      \\
\hline 
10000   & 0.00043 & 0.00044   &0.00129    &  0.00018    & 0.00018   &    0.00040   \\
\hline 
\hline
\   & \multicolumn{3}{c||}{$\mathcal{N}(2,1)$ } \\
\cline{1-4 }
n   & $r_{or}$  & $r_{ad }$  & $r_{eq}$ \\
\cline{1-4} 
1000 & 0.00089&0.00090   &    0.00419     \\ 
\cline{1-4}
5000 &  0.00056&    0.00057  &        0.00217     \\
\cline{1-4}
10000&0.00052 &  0.00052 &      0.01533         \\  
\cline{1-4}
\end{tabular}
\caption{Estimation of the distributional density, with $\mathcal{U}([0,2])$-distributed observation distances   }
\end{table}

\textbf{Discussion:} It is surprising to realize that for medium to large sample sizes, there is no visible difference in the performance of the oracle estimator and the estimator with data driven choice of the bandwidth and the weights, whereas the procedure with all weights set equal to $1$ performs substantially worse. 
\section{Proofs}\label{Abschnitt_Beweise}
\subsection{Preliminaries}
\begin{lemma} \label{Lemma_Momente} For arbitrary $m \in \N$, there exists a positive constant $C$ depending on $m$ such that 
\begin{equation}
\ebig{\abs{\hat{q}(u) - q(u)    } ^{2m} }  \leq C 
\big( \sum_{j=1}^{n} \Delta_j^2 |w_j(u)|^2   \big)^m = C \sigma(u)^{2m}.
\end{equation}
\end{lemma}
\begin{proof} 
By the   Rosenthal inequality (see, e.g. \cite{Ibragimov_Ros}) and the fact that $|\hat{\phi}_j-\phi_j|$  is bounded by $2$, there exist constants $C'$ and $C$ depending on $m$ such that 
\begin{align}
\notag \phantom{\leq}&  \ebig{\abs{ \hat{q}(u) - q(u)    } ^{2m} }    \\
\leq &  C' \max\Big\{   \bigg(\sum_{j=1}^{n} \Delta_j^2 |w_j|^2 
\ebig{ \big|\hat{\phi}_j (u) - \phi_j(u) \big|^2  }  \bigg)^{m} , \sum_{j=1}^{n} \Delta_j^{2m} |w_j|^{2m} 
\ebig{ \big|\hat{\phi}_j (u) - \phi_j(u) \big|^{2m}  }   \Big\}  \\
\leq &  C \bigg(\sum_{j=1}^{n} \Delta_j^2 |w_j|^2 
  \bigg)^{m}  = C \sigma(u)^{2m}.
\end{align}
This is the desired result. 
\end{proof}
\begin{lemma}\label{Lemma_Neumann} For arbitrary $m\in \N$,  there exists a constant $C$ depending on $\kappa$, $\Dmax$  and $m$ such that 
\begin{align}
\ebigg{\bigg| \frac{1}{q(u)}- \frac{1}{\tilde{q}(u) } \bigg|^m   }      \leq C \min \bigg\{ \frac{1}{|q(u)|^m}, \frac{1}{|q(u)|^{\frac{3m}{2} }  } \bigg\}               
\end{align}

\end{lemma}
\begin{proof}
Consider first the case where $|q(u)| \leq  2\sigma(u)$.  Using the fact that, by definition of $1/\tilde{q}$ and the assumption on $q$,
\[
 \frac{1}{|\tilde{q}(u)| }\leq \frac{1}{\sigma(u) }< \frac{2}{|q(u)|}, 
\]
  as well as  the estimate 
\begin{align}
  & \sigma(u) =\Big( \sum_{j=1}^{n} \Delta_j^2 |w_j(u)|^2   \Big)^{\frac{1}{2}  }  \leq \Big(\Dmax
\sum_{j=1}^{n}  \Delta_j |w_j(u)|^2   \Big)^{\frac{1}{2}  }   =   
\Big(\Dmax\sum_{j=1}^{n}  \Delta_j |\phi_j(u)|^2   \Big)^{\frac{1}{2}  } \\
 =& \Dmax^{\frac{1}{2} } q(u)^{\frac{1}{2} }, \label{Neumann_0}
\end{align}
 we find that 
\begin{align}
\ebigg{\bigg| \frac{1}{q(u)}- \frac{1}{\tilde{q}(u) } \bigg|^m   }
\leq 3^m \frac{1}{|q(u)|^m }  \leq  6^m \frac{\sigma(u)^m }{|q(u)|^{2m}  }
\leq    \frac{(6\sqrt{\Dmax})^m}{|q(u)|^{\frac{3m}{2} } }.
\end{align}  
Next, assume that $|q(u)|>2\sigma(u)$ and, in addition, $\sigma(u)\geq \kappa$.   We may then use the decomposition 
\begin{align}
\ebigg{\bigg| \frac{1}{q(u)}- \frac{1}{\tilde{q}(u) } \bigg|^m   }
= \frac{1}{|q(u)|^m }  \PP\big(\big\{|\hat{q}(u)| <\sigma(u)  \big\}\big)   + \ebigg{\bigg| \frac{1}{q(u)}- \frac{1}{\hat{q}(u) } \bigg|^m  \ind{\{ |\hat{q}(u)| \geq \sigma(u) \} } }. 
\end{align} 
By the Markov inequality, 
\begin{align}
 &  \PP\big(\big\{  |\hat{q}(u)| <\sigma(u)  \big\} \big)  \leq \PP\Big(\Big\{ |q(u) - \hat{q}(u)| > \frac{|q(u)|}{2}  \Big\} \Big)   \leq 2^m \frac{  \E[    |q(u) - \hat{q}(u)|^m       ]     }{  |q(u)|^m    }  \\
\leq & 2^m \frac{\sigma(u)^{m}  }{|q(u)|^m } \leq (2 \sqrt{\Dmax} )^m \frac{1  }{|q(u)|^{\frac{m}{2} } }.
\end{align}
On the other hand,  using Lemma \ref{Lemma_Momente}  and formula \eqref{Neumann_0}  we can estimate for a constant $C$ depending on $m$,  
\begin{align}
&  \ebigg{\bigg| \frac{1}{q(u)}- \frac{1}{\hat{q}(u) } \bigg|^m  \ind{\{ |\hat{q}(u)| \geq \sigma(u) \} } }
\leq 2^m \Big( \frac{ \E{   |q(u) - \hat{q}(u)|^m       }     }{  |q(u)|^{2m}       }
+  \frac{ \E{   |q(u) - \hat{q}(u)|^{2m}       }     }{  |q(u)|^{2m}   \sigma(u)^m       }  \Big) \\
\leq & C \frac{\sigma(u)^m }{|q(u)|^{2m}  } \leq C\min\Big\{ \frac{1}{|q(u)|^m},
\frac{\sigma(u)^m}{|q(u)|^{2m} }   \Big\}  \leq C \Big\{    \frac{1}{|q(u)|^m},
\frac{\sqrt{\Dmax}^m}{|q(u)|^{\frac{3m}{2}} }    \Big\} .   \label{Neumann_1}
\end{align}
Finally, consider the case where $|q(u)|> 2\sigma(u)$ and $\sigma(u)<\kappa$.  We can then decompose 
\begin{align}
\ebigg{\bigg| \frac{1}{q(u)}- \frac{1}{\tilde{q}(u) } \bigg|^m   }
= \frac{1}{|q(u)|^m }  \PP\big(\big\{|\hat{q}(u)| <\kappa \big\}\big)   + \ebigg{\bigg| \frac{1}{q(u)}- \frac{1}{\hat{q}(u) } \bigg|^m  \ind{\{ |\hat{q}(u)| \geq \kappa \} } }.  \label{Neumann_2}
\end{align} 
We distinguish between two different sub-cases.

a)  In addition to $|q(u)|> 2\sigma(u)$ and $\sigma(u)<\kappa$, $|q(u)|\leq  2\kappa$.  We may then estimate 
\begin{align}
\frac{1}{|q(u)|^m }  = (2\kappa)^{\frac{m}{2} } \frac{1}{|q(u)|^m (2\kappa)^{\frac{m}{2}}  }
\leq  (2\kappa)^{\frac{m}{2} } \frac{1}{|q(u)|^{\frac{3m}{2}  }  } 
\end{align}
and, moreover, 
\begin{align}
\bigg| \frac{1}{q(u)}- \frac{1}{\hat{q}(u) } \bigg|^m  \ind{\{ |\hat{q}(u)| \geq \kappa \} }
\leq\frac{3^m}{|q(u)|^m}\leq    \frac{ (3 \sqrt{2\kappa} )^m }{|q(u)|^{\frac{3m}{2}} } .
\end{align}
Along with formula \eqref{Neumann_2}, this yields the desired result for the sub-case. 

b) $|q(u)|> 2\sigma(u)$, $\sigma(u)<\kappa$ and, in addition, $|q(u)|> 2\kappa$.  In this case, another application of the Markov inequality and Lemma \ref{Lemma_Momente} gives  for a constant $C$ depending on $m$, 
\begin{align}
 \PP\big(\big\{|\hat{q}(u)| <\kappa \big\}\big)   \leq \PP\Big(\Big\{|q(u) - \hat{q}(u)| > \frac{|q(u)|}{2} \Big\}\Big)  \leq C \frac{\sigma(u)^m }{|q(u)|^{m}  }  \leq C \sqrt{\Dmax}  \frac{1}{|q(u)|^{\frac{3m}{2}}  }.  
\end{align}
Next, arguing along the same lines as in formula \eqref{Neumann_1}, we find that 
\begin{align}
& \ebigg{  \bigg| \frac{1}{q(u)}- \frac{1}{\hat{q}(u) } \bigg|^m  \ind{\{ |\hat{q}(u)| \geq \kappa \} }   }
\leq C  \Big(   \frac{\sigma(u)^m }{|q(u)|^{2m}  } + \frac{\sigma(u)^{2m}  }{  |q(u)|^{2m}  \kappa^m } \Big)  \leq  C \frac{\sigma(u)^{m}  }{|q(u)|^{2m} }   \\
\leq  & C  \min\Big\{  \frac{2^m}{|q(u)|^{m} },    \frac{2^{m+1}  \Dmax^{\frac{m}{2} } }{|q(u)|^{\frac{3m}{2} } } \Big\}.
\end{align}
This completes the proof.
\end{proof}

Recall that $X_+$ and $X_{-}$ are  independent  L\'evy processes without drift and with jump measures  $\nu_+(\d x)=\nu |_{(0,\infty)}(\d x)$ and $\nu_{-}(\d x)= \nu|_{(-\infty,0)}(\d (\m x) )$

\begin{lemma}  \label{Lemma_Ableitungen} For arbitrary $\Delta\geq 0$  and $m\in \N$,   
\begin{align}
\E{ |X_\Delta|^m }  \leq    \max\{\Delta, \Delta^m\} \big(  \E{(X_+)_{1}^m }+
\E{(X_{-})_{1}^m  }    \big)   = \max\{\Delta, \Delta^m\} C^{\pm}_{m}
\end{align}

\end{lemma}
\begin{proof} Since $X_\Delta$ has the same distribution as $(X_+)_\Delta - (X_-)_\Delta$, we have 
\begin{equation}
\E{ |X_\Delta|^m  }=\E{ |(X_+)_{\Delta} - (X_{-})_{\Delta} |^m }  \leq \E{ \max \{ (X_+)_{\Delta}^m , (X_{-})_{\Delta}^m \} }\leq \E{(X_+)_{\Delta}^m}  +\E{(X_{-})_{\Delta}^m}.
\label{dec pos neg}
\end{equation}
Let $\mathcal{A}:=\{\alpha \in \{1,...,m\}^m:   \sum_{\ell=1}^{m}  \ell \alpha_\ell =m \}$.
By Fa\`a di Bruno's theorem, there exist universal constants $c_\alpha, \alpha\in \mathcal{A}$ such that 
\begin{align}
\E{(X_+)_{\Delta}^m }  = i^{\m m}  \phi_{{(X_+)_{\Delta}}}^{(m)} (0)   = i^{\m m}  \sum_{\alpha  \in \mathcal{A}}
c_\alpha   \prod_{\ell =1}^{m}   \big( \Delta \Psi^{(\ell)}(0)\big)^{ {\alpha_{\ell} }  }
=  \sum_{\alpha  \in \mathcal{A}}
c_\alpha   \prod_{\ell =1}^{m}   \Big(\Delta \int  x^{\ell}\nu_+(\d x)\Big)^{ {\alpha_{\ell} }  },.
\end{align}
Since  $\nu_+$ is supported on the positive half axes, this implies 
\begin{align}
\E{ (X_+)_{\Delta}^m }   \leq  \max\{\Delta, \Delta^m\} \sum_{\alpha  \in \mathcal{A}}
c_\alpha   \prod_{\ell =1}^{m}   \Big( \int  x^{\ell}\nu_+(\d x)\Big)^{ {\alpha_{\ell} }  }
=\max\{\Delta, \Delta^m\}   \E{(X_+)_1^{m}  }.
\end{align}
Arguing along the same lines gives $\E{ (X_{-})^m_{\Delta} }\leq \max\{\Delta, \Delta^m\}
\E{(X_{-})_1^{m}  }$. Combining this with formula \eqref{dec pos neg}, gives the statement of the Lemma. 
\end{proof}
\begin{lemma} \label{Lemma_p}For any $m\in \N$, there exists a constant $C$ depending on $m$, $\DMax$ and $C_{2m}^{\pm}$  such that 
\begin{align}
& \Ebig{  \big|\hat{p}_n(u) - p(u) \big|^{2m}  }
\leq C\max \Big\{|q(u)|^m, |q(u)|  \Big\}.
\end{align}

\end{lemma}
\begin{proof} Another application of the Rosenthal inequality,  along with the fact that $|w_j|\leq 1$ and Lemma \ref{Lemma_Ableitungen} gives,  for  constants $C_k, \, k=1, ...,3$ depending only on $m$,  and some $C$ depending on $m$, $\DMax$ and $C_{2m}^{\pm}$, 
\begin{align}
& \Ebig{  \big|\hat{p}(u) - p(u) \big|^{2m}  }  \\
\leq  & C_1 \max\Big\{   \Big( \sum_{j=1}^{n}|w_j(u)|^2  \E{|\hat{\phi}_j'(u)-\phi_j'(u)|^2} \Big)^m,
\sum_{j=1}^{n} |w_j(u)|^{2m} \E{  |\hat{\phi}_j'(u)-\phi_j'(u)|^{2m} } \Big\}  \\
\leq & C_2  \max\Big\{   \Big( \sum_{j=1}^{n}|w_j(u)|^2  \E{|Z_j|^2} \Big)^m,
\sum_{j=1}^{n} |w_j(u)|^{2m} \E{  |Z_j|^{2m} } \Big\}  \\
\leq & C_3 C_{2m}^{\pm} \max\Big\{   \Big( \sum_{j=1}^{n}|w_j(u)|^2  \max\{
\Delta_j,\Delta_j^2 \}   \Big)^m, \sum_{j=1}^{n}|w_j(u)|^{2m} \max\{\Delta_j , \Delta_j^{2m} \} \Big\}  \\
\leq & C_4 C_{2m}^{\pm} \DMax^{2m-1} \max\Big\{ \Big( \sum_{j=1}^{n} \Delta_j |w_j(u)|^2 \Big)^m 
, \sum_{j=1}^{n}  \Delta_j |w_j(u)|^{2}  \Big\}  =  C\max \big\{|q(u)|^m, |q(u)|  \big\} .
\end{align}
This is the desired result 
\end{proof}
\subsection{Estimation of the jump dynamics}
\subsubsection{Non-asymptotic upper bound}
\textbf{ Proof of Theorem \ref{Hauptsatz_Risiko}  } 
We can estimate 
\begin{equation}
\E{\| g- \hat{g}_{h,n}\|_{{\lk^2}}^2 } \leq 2 \| g- \kf_h\ast g \|_{{\lk^2(\Omega)}}
+2\E{ \|\kf_h\ast g - \hat{g}_{h,n}\|_{{\lk^2}}^2}.
\end{equation}
The Plancherel formula gives 
\begin{eqnarray}
 \E{ \|\kf_h\ast g - \hat{g}_{h,n}\|_{{\lk^2}}^2 }
= \frac{1}{2\pi } \int  |\Fourier \kf_h(u)|^2\E{ \Big|{\frac{ \hat{p}(u)   }{\tilde{q}(u)   }   - \frac{ p(u)   }{q(u)   }              }\Big|^2} \d u. 
\end{eqnarray}
Now, we have the inequality 
\begin{eqnarray}
\Big|{\frac{ \hat{p}(u)   }{\tilde{q}(u)   }   - \frac{ p(u)   }{q(u)   }              }\Big|^2
\leq 2\left(\frac{ | \hat{p}(u) -p(u)     |^2    }{ |\tilde{q}(u)|^2  }   +|p(u)|^2 
\Big|\frac{1}{q(u)}- \frac{1}{\tilde{q}(u)}             \Big|^2\right).  \label{HS_1}
\end{eqnarray}
Consider the first summand on the right hand side of  formula \eqref{HS_1}.  We start by considering  $|q(u)|<1$.  Then, using the fact that $|\tilde{q}|$ is bounded from below by $\kappa$ as well as Lemma \ref{Lemma_p}, we find that 
\begin{align}
\ebigg{ \frac{ | \hat{p}(u) -p(u)     |^2    }{ |\tilde{q}(u)|^2  }             }
\leq   \frac{ \E{ | \hat{p}(u) -p(u)     |^2  }    }{\kappa^2 |q(u)|^2  }
\leq C \frac{1}{|q(u)| },
\end{align}
with some $C$ depending on $\Dmax, C^{\pm}_2$ and $\kappa$. 

Next, consider the case where $|q(u)|>1$. With the Cauchy-Schwarz inequality, Lemma   \ref{Lemma_p}  and Lemma \ref{Lemma_Neumann}, for some $C$ depending on $\DMax$, $C_{4}^{\pm}$ and $\kappa$, 
\begin{eqnarray}
\E{\frac{ | \hat{p}_n(u) -p(u)     |^2    }{ |\tilde{q}_n(u)|^2  }    }
\leq \mathbb{E}^{\frac{1}{2}}\left[| \hat{p}_n(u) -p(u)     |^4   \right]
\mathbb{E}^{\frac{1}{2}} \left[\frac{1}{|\tilde{q}_n(u)|^4}\right]\leq C   \frac{1}{|q(u)|}.
\end{eqnarray}
Thanks to  Lemma   \ref{Lemma_p}  and Lemma \ref{Lemma_Neumann}, for some $C$ depending on $\DMax$, $C_{4}^{\pm}$ and $\kappa$, 
\begin{eqnarray}
\mathbb{E}^{\frac{1}{2}}\left[| \hat{p}(u) -p(u)     |^4   \right]
\mathbb{E}^{\frac{1}{2}} \left[\frac{1}{|\tilde{q}(u)|^4}\right]
\leq C   \frac{1}{|q(u)|}.
\end{eqnarray}
Consider now the second summand on the right hand side of formula \eqref{HS_1}. An application of Lemma \ref{Lemma_Neumann}, along with the fact that 
\begin{align}
|\Psi'(u)|= \bigg|  \int x e^{iux}  \nu(\d x) \bigg| \leq \int |x|\nu(\d x)= \E[(X_+)_1]+\E[(X_-)_1] =   C_{1}^{\pm}, 
\end{align}
 yields  for some $C$ depending on $\DMax$ and $\kappa$, 
\begin{eqnarray}
|p(u)|^2 
\E{\Big| \frac{1}{q(u)}- \frac{1}{\tilde{q}(u)}              \Big|^2}
\leq C \frac{ |p(u)|^2 }{|q(u)|^{3} }   =  C \frac{|\Psi'(u)|^2}{|q(u)| }
\leq C (C^{\pm}_1)^2 \frac{1}{|q(u)|}.
\end{eqnarray}
This completes the proof.   \hfill $\Box$
\subsubsection{Rate results: Upper bounds }

\textbf{Proof of Lemma \ref{Lemma_Bias}.}We can estimate 
\renewcommand{\norm}[1]{\left\| #1\right\|_{\lk^2(\R)}^2 }
\begin{eqnarray}
\norm{ g- \kf_h\ast g}\leq 2\norm{\tilde{g} - \kf_h \ast \tilde{g} }  +
2\norm{\kf_h\ast \left( \tilde{g} -  g\right)  }.
\end{eqnarray}
A standard Taylor series argument implies for some constant $C$ depending on the choice of $\kf$, on $L$, $R$ and  $\omega_2- \omega_1$, 
\[
\norm{\tilde{g} - \kf_h \ast \tilde{g} }  \leq C  h^{2a}. 
\]

It remains to consider the second summand. Let $\delta:= \min\{\omega_1 -  d_1, d_2 - \omega_2\}$.  Using the fact that $\tilde{g}- g$ vanishes on $D$, as well as \eqref{Abfall_Kern},  we have for arbitrary $x\in \Omega$: 
\begin{eqnarray}
&\phantom{\leq}& \abs{ \kf_h \ast \left( \tilde{g} -  g    \right) (x)  }=\abs{ \int  \frac{1}{h} \kf \left(\frac{ x-y}{h } \right)\left( \tilde{g} -  g    \right) (y) \d y }\\
& \leq & \sup_{|z|\geq \delta}  \frac{1}{h}\abs{ \kf\left(  \frac{z}{h} \right)}
\left( \left\|\tilde{g}\right\|_{\lk^1(\R) } + \left\|{g}\right\|_{\lk^1(\R) }  \right)
\leq \delta^{-2a-1} \left( \left\|\tilde{g}\right\|_{\lk^1(\R) } + \left\|{g}\right\|_{\lk^1(\R) }  \right) h^{2a}. 
\end{eqnarray}
This completes the proof. \hfill $\Box$

\subsubsection{Minimax lower bounds}
\textbf{Proof of theorem \ref{Satz_untere_Schranken}}
The minimax lower bounds are established by looking at a decision problem between an increasing finite number $M$ of alternatives, see Theorem 2.5 in Tsybakov\  (2003).  The construction of the alternatives essentially  follows the proof of Theorem 4.4 in Neumann and Rei\ss \ (2009). For this reason, we only sketch the most important  steps and omit some of the technical details. 

Let 
\[
\eta(x) =  b |x| e^{-\lambda |x| },  \  x\in \R 
\]
be the L\'evy density of a  symmetric bilateral Gamma-distribution with parameters 
$b=\beta/2$ and with $\lambda>0$ to be appropriately chosen.  
Let $g(x):=x \eta (x)$.   By $\PP_{0,j}$ we denote the infinitely divisible distribution with characteristic function  
\[
\phi_{0,j}(u) = \exp\Big(\Delta_j \int (e^{iux}-1) \eta(x) \d x \Big).
\]

Since $\eta$ is infinitely differentiable away from zero and $\eta(x)\sim |x|^{-1}$, $\eta \to 0$, one can, with the appropriate choice of $\lambda$,  always guarantee that $g$ or $g|_D$ is contained in the prescribed class of globally or locally regular functions. 

By   $f_{\mathcal{N}(\mu,\sigma)}$ we denote the density of a $\mathcal{N}_{\mu,\sigma}$-distribution.  For  $n\in \N$, a positive integer $m_n$  to be appropriately chosen  and some universal positive constant  $d>0$, we introduce  the following perturbations of $\eta$, 
\begin{align}
h_{n,j}(x):= d   K_n2\sin(  (m_n+j)  x)f_{ 
\mathcal{N}(0,1)}(x), \   j= 0, \cdots, m_n - 1 .
\end{align}  
When locally H\"older regular functions are considered, $K_n:= m_n^{-a-\frac{1}{2} }$. For functions $g$ having a Fourier transform which decays at the rate $|u|^{-1}$, $K_n:=m_n^{-1}$. 
For $S\subseteq \{0,\cdots, m_n -1\}$,  we define 
\begin{align}
\eta_{S,n}(x):= \eta(x)+\sum_{j\in S} h_{n,j}(x). 
\end{align}
Moreover, $g_{S,n}(x):= x \eta_{S,n}(x)$.  Using the fact that the multiplication with a sine- function corresponds to a shift in the Fourier domain, we find that 
\begin{align}
 & \Fourier(x  h_{n,j}(x) ) (u) \\
 =&  \frac{d K_n}{2}   \Big[
(u-(m_n+j) ) \exp\Big(- \frac{1}{2} (u - (m_n+j) )^2 \Big) -  (u + (m_n+j) )\exp\Big(-  \frac{1}{2} (u+m_n+j) ^2 \Big].
\end{align}
On the other hand, arguments presented in Neumann and Rei\ss \ (2009) ensure that the characteristic function of the perturbed distribution has the same decay behavior as  $\phi_{0,j}$.

From there, we derive that the perturbed functions $g_{S,n}$ are still in the prescribed classes of globally or locally regular functions.

Let us estimate, for $S_1,S_2 \subseteq \{0,\cdots, m_n -1 \}$, the $\lk^2$-distance between $g_{S_1,n}$ and $g_{S_2,n}$. We start by observing that, using elementary calculus, the decay of the normal density as well as the periodicity of the sine-function,  
\begin{align}
\int_\Omega |g_{S_1,n}(x) - g_{S_2,n}(x)|^2 \d x
 &= \int_\Omega  \Big|  \sum_{j\in S_1 \vartriangle S_2}  x h_{h,j}(x) \Big|^2 \d x\\
 &\geq  c \int \Big| \sum_{j\in S_1 \vartriangle S_2} x h_{h,j}(x) \Big|^2 \d x  = c \int |g_{S_1,n}(x) - g_{S_2,n}(x)|^2 \d x
 \end{align}
 for a small enough positive constant $c$. By the Plancherel formula and the construction of the $h_{n,j}$, 
 \begin{align}
 & \int |g_{S_1,n} (x) -g_{S,n}(x)|^2 \d x
= \frac{1}{2\pi} \int |\sum_{j\in S_1\vartriangle S_2}  \Fourier (x h_{n,j}(x)) (u)|^2 \d u   \\
\geq  &  \frac{1}{2\pi} \sum_{j\in S_1\vartriangle S_2}\int |\Fourier (x h_{n,j}(x)) (u)|^2 \d u
\geq C |S_1\vartriangle S_2 | K_n^2.
\end{align}  
For some universal positive constant $C$. 

Next, we calculate the Kullback-Leibler divergence between the competing measures. 
\begin{align}
\KL( \PP_{S,n} |  \PP_{0,n}  ) = \sum_{j=1}^{n}   \KL( \PP_{S,n,j}|\PP_{0,n,j} ) 
\leq \sum_{j=1}^{n}  \chi(\PP_{S,n,j}, \PP_{0,n,j} ) .
\end{align}
We may again argue  along the same lines as in the proof of Theorem 4.4 in Neumann and Rei\ss \  (2009)  to find that for some positive constant $C'$ depending on $d$,  
\[
\sum_{j=1}^{n}  \chi(\PP_{S,n,j}, \PP_{0,n,j} )   \leq C'|S| \sum_{j=1}^{n} \Delta_j   (1+m_n)^{-2\Delta_j \beta}      K_n^2. 
\]

For $m_n\geq 8$, the Varshamov-Gilbert Lemma (see chapter 2 in Tsybakov (2003)) guarantees the existence of $M\geq \sqrt[8]{2}^{m_n}$ different subsets $S_0, \cdots, S_{M-1}$ of $\{0,\cdots, m_n-1\}$, including the empty set  $S_0=\emptyset $, such that 
\[
|S_k \vartriangle S_\ell|   \geq m_n/8   \quad   \forall 0 \leq k< \ell \leq M-1. 
\]

Theorem  2.5 in Tsybakov (2003) then implies that, with $K_n$ implicitly defined by $K_n\sum_{j=1}^{n} \Delta_j (1+|m_n|)^{-2\beta}  =1$,
\begin{align} 
\limsup_{n\to \infty}\,  \inf_{\tilde{g}_n}\,   \sup_{\check{g} \in \mathcal{G}  } \, 
\mathbb{E}_{g} \big[  \|\check{g}- g \|_{\lk^2}^2  \big] K_n^{-2} m_n \geq 0. 
\end{align}
Here $\mathcal{G}$ may either stand for the class $\mathcal{G}_{\text{\tiny pol}  }(\beta, C_\phi,c_\phi,C_g,C)$ or $\mathcal{G}(a,D,L,R, C_1, \beta, C_2, C_3)$ and $\lk^2$ denotes the $\lk^2$ norm on the whole real axes or on $\Omega$, respectively. This completes the proof. 
\hfill  $\Box$

\subsection{Estimating the distributional density}
\subsubsection{Upper risk bound}
Before proving the main result of the section, we formulate an auxiliary result.  
The proof of the Lemma is postponed to the appendix.
\begin{lemma} \label{Hauptlemma} For any $m\in  \N$,  there exists a constant $C$ depending on $m$, $\Dmax$ and $C_{4m}^{\pm}$ such that 
\begin{align}
& \Ebig{ \Big|\intu \frac{\hat{p}(x)}{\tilde{q}(x) } -\frac{p(x)     }{ q(x) }  \d x \Big|^{2m} } 
\leq  C  \Big(C_{\Psi,1}  \intu \frac{1}{q(x) }  \d x \Big)^m \Big( \max\Big\{ 1, C_{\Psi,2}^2 \intu \frac{1}{q(x) }  \d x \Big\} \Big)^{m} ,  
\end{align}
with 
\[
C_{\Psi,1}= (\|\Psi'\|_{\lk^2,u}^2 +\|\Psi''\|_{\lk^1,u})\vee 1   \quad \text{and}  \quad  
C_{\Psi,2} = \|\Psi'\|_{\infty,u} \vee 1.
\]
\end{lemma}
\textbf{Proof of Theorem \ref{Schranken_DS} } Parseval's identity, along with  the fact that $\kf$ is compactly supported  gives 
\begin{align}
\ebig{ \|f - \hat{f}_{h,n} \|_{\lk^2}^2  }  \leq 2 \| f- \kf_h \ast f \|_{\lk^2}^2 +\frac{1}{\pi}  \inth \ebig{ |\hat{\phi}(u) - \phi(u)|^2 }\d u
\end{align}
In the sequel, we use the short notation  notation
\begin{align}
\Delta(u) :=   \intu   \left( \frac{\hat{p}(z) }{\tilde{q}(z) }
-\frac{p(z)}{q(z) } \right)   \d z.
\end{align}
Since we have, by definition of $\hat{\phi}$,   $|\hat{\phi}(u) - \phi(u)|\leq |\check{\phi}(u) - \phi(u)|$, as well as  $|\hat{\phi}(u) - \phi(u)|\leq 2$, we can estimate 
\begin{align}
|\hat{\phi}(u) - \phi(u)|^2  \leq   |\check{\phi}(u) - \phi(u)|^2 \mathds{1}_{ \{ |\Delta(u)| \leq 1 \}}
+4 \mathds{1}_{ \{ |\Delta(u)| >1 \}}
\end{align}
Using the definition of $\check{\phi}$, as well as the fact that   $|1- \exp(z)| \leq 2 |z|$ holds for $|z|\leq 1$, we can continue by estimating for arbitrary  $m\in \N$: 
\begin{align}
& |\check{\phi}(u) - \phi(u)|^2 \mathds{1}_{ \{ |\Delta(u)| \leq 1 \}}
+4 \mathds{1}_{ \{ |\Delta(u)| >1 \}}  
\leq   |\phi(u)|^2  |1- \exp(\Delta(u))|^2 \mathds{1}_{\{|\Delta(u)| \leq 1\}  }    + 4  \mathds{1}_{ \{|\Delta(u)| >1 \}  }   \\
\leq &  4 |\phi(u)|^2 |\Delta(u)|^2   +4  |\Delta(u)|^{m }.
\end{align}
By Lemma \ref{Hauptlemma},  for some $C$ depending on $\DMax$ and on $\E{X_{4m}^{\pm} }$, 
\begin{align}
& \int_{-1/h}^{1/h}  | \phi(u) |^2\ebig{ |\Delta(u)|^2 \d u  }
       \leq   C C_{\Psi,1}  \inth |\phi(u)|^2  \intu \frac{1}{q(z) } \d z  \d u
\end{align}
and 
\begin{align}
& \int_{-1/h}^{1/h} \ebig{ |\Delta(u)|^m  } \d u \leq C  \inth \Big( C_{\Psi,1} \intu \frac{1}{q(x)}  \d x \Big)^m  \Big( \max\Big\{ 1, C_{\Psi,2}  \intu \frac{1}{q(z)} \d z  \Big\}\Big)^m \d u.
\end{align}
This completes the proof.    \hfill   $\Box$




\section*{Appendix}
 \textbf{Proof of Lemma \ref{Hauptlemma}}
We use the estimate 
\begin{align}
& \Big|\intu \frac{\hat{p}(x)}{\tilde{q}(x) } -\frac{p(x)     }{ q(x) }  \d x \Big|^{2m}  \\
 \label{Zerlegung_Hauptlemma}\leq & 3^{2m-1}  \Big( \Big|\intu \frac{ \epx -p(x)     }{ q(x) }  \d x \Big|^{2m} \hspace*{-0.2cm}
+ \Big|\intu ( \epx -p(x)   ) \Rem (x)  \d x \Big|^{2m}
\hspace*{-0.2cm}+  \Big|  \intu p(x) \Rem(x)   \d x  \Big|^{2m}    \Big). 
\end{align}
We bound, successively,  the expected value of each of the three terms in the second line of formula \eqref{Zerlegung_Hauptlemma}. 

\    \\
- Consider the first summand: By the Rosenthal inequality, 
\begin{align} \label{Rosen}
 &\Ebig{ \Big|\intu \frac{ \epx -p(x)     }{ q(x) }  \d x \Big|^{2m} } \\
\leq &
\max \Big\{  \sum_{j=1}^{n}  \Ebig{  \Big| \lintu  \frac{ w_j(x) (\phi'_j- \hat{\phi}_j')(x) }{q(x) }  \d x\Big|^{2m}}, \Big(\sum_{j=1}^{n}  \Ebig{  \Big| \lintu  \frac{ w_j(x) (\phi'_j- \hat{\phi}_j')(x) }{q(x) }  \d x\Big|^2} \Big)^{m}\Big\} 
\end{align}
For any integer  $k\geq 2$,
\begin{align}
& 2^{\m k} \Ebig{  \Big| \lintu  \frac{ w_j(x) (\hat{\phi}_j'-\phi'_j )(x) }{q(x) }  \d x\Big|^{k}}
\leq  \Ebig{  \Big| \lintu  \frac{ w_j(x) \hat{\phi}_j'(x)}{q(x) }  \d x\Big|^{k}                              }    \\
= & \int_{[0,u]^k}\limits\prod_{\ell =1}^{k} \frac{  w_j\big( (\m 1)^{\ell +1}  x_\ell  \big)   }{  q\big( (\m 1)^{\ell +1}  x_\ell  \big) }  \ebig{  (i Z_j)^{k}  \exp\big( i \big( \sum_{\ell =1}^{k} (\m 1)^{\ell +1} x_\ell \big) Z_j     \big)     }  
\d x_1 \cdots \d x_k  \\
= & \int_{[0,u]^k}\limits\prod_{\ell =1}^{k} \frac{  w_j\big( (\m 1)^{\ell +1}  x_\ell  \big)   }{  q\big( (\m 1)^{\ell +1}  x_\ell  \big) }  \phi_j^{(k)} \big( \sum_{\ell =1}^{k} (\m 1)^{\ell +1} x_\ell \big)  
\d x_1 \cdots \d x_k.
\end{align}

Repeated applications of the integration by parts formula give for some constant  $C$ depending on $k$, 
\begin{align}
 &\int_{[0,u]^k}\limits\prod_{\ell =1}^{k} \frac{  w_j\big( (\m 1)^{\ell +1}  x_\ell  \big)   }{  q\big( (\m 1)^{\ell +1}  x_\ell  \big) }  \phi_j^{(k)} \big( \sum_{\ell =1}^{k} (\m 1)^{\ell +1} x_\ell \big)  
\d x_1 \cdots \d x_k \\
\leq & C \hspace*{-0.1cm} \sum_{m \leq k-2}  \Big(   \max \Big\{\frac{1}{q(u)}, \frac{1}{q(0) } \Big\}^m \sup_{t\in\R}\intu \intu   \frac{|\phi_j''  \big(t+x_1 \big) w_j(x_1) w_j(x_2)|}{|q(x_1)q(x_2)|}  \d x_1  \d x_2 \\
& \phantom{ C \hspace*{-0.1cm} \sum_{m \leq k-1}   }     \underset{{[0,u]^{k-m-2}}   }{\int} 
\prod_{\ell=2}^{k-m-2}
\Big| \Big( \frac{w_j}{q} \Big)'(x_\ell)  \Big|   \d   x_3  \cdots   \d  x_{k-m-1   }   \Big).
\end{align}
First, using $\phi_j''= \Delta_j \Psi'' \phi_j +  \Delta_j^2 (\Psi')^2  \phi_j$ and the Cauchy-Schwarz inequality, we obtain 
\begin{align}
&\intu \intu   \frac{|\phi_j''  \big(t+x_1 \big) w_j(x_1) w_j(x_2)|}{|q(x_1)q(x_2)|}  \d x_1  \d x_2 \leq ( \|\Psi'\|_{\lk^2}^2 +\|\Psi''\|_{\lk^1} )    \DMax \intu \frac{ \Delta_j |w_j(x)|^2}{q(x)^2 } \d x.
\end{align}
On the other hand,
\begin{align}
q'(x) = \sum_{j=1}^{n} \Delta_j  (w_j'(x) \phi_j(x) +w_j(x) \phi_j'(x) )   =2 \sum_{j=1}^{n} 
\Delta_j^2 \Psi'(x) |\phi_j(x)|^2 
\end{align}
and consequently, 
\begin{align}
\Big| \Big( \frac{w_j}{q} \Big)'(x) \Big|  \leq \frac{w_j'(x)}{q(x) } + 2\Dmax \frac{\Psi'(x) w_j(x)}{q(x)}
=(1+2\Dmax) \frac{\Psi'(x) w_j(x)}{q(x)}.
\end{align}
This implies 
\begin{align}
\underset{{[0,u]^{k-m-2}}   }{\int} 
\prod_{\ell=3}^{k-m-2}
\Big| \Big( \frac{w_j}{q} \Big)'(x_\ell)  \Big|   \d x_3  \cdots \d x_{k-m-2   }
\leq \Big( 3 \DMax \|\Psi'\|_{\infty} \intu \frac{1}{q(x)}\d x \Big)^{k-m-2} .
\end{align}
Finally,  using the fact that $q$ is bounded above by $T$ and $q(0)=T$,   we find that for any $u\geq1$, 
\begin{align}
\max \Big\{\frac{1}{q(u)}, \frac{1}{q(0) } \Big\}  \leq \frac{2}{q(0)}+ \Dmax \intu \frac{|\Psi'(x)|}{q(x) }  \d x
\leq \DMax (|\Psi'\|_{\infty} +2)   \intu \frac{1}{q(x) }  \d x.
\end{align}
Putting the above together, we have shown that for any integer $k\geq 2$, there exists a constant $C$ depending on  $\Dmax$ and $k$ such that 
\begin{align}
  & \Ebig{  \Big| \lintu  \frac{ w_j(x) \hat{\phi}_j'(x)}{q(x) }  \d x\Big|^{k}     }   \leq  C   \Big(   C_{\Psi,2}\intu \frac{1}{q(x) }  \d x \Big)^{k-2}  
C_{\Psi,1}\intu  \frac{\Delta_j |w_j(x)|^2 }{q(x)^2 }  \d x.
\end{align}
From there, we can conclude that 
\begin{align}
& \sum_{j=1}^{n} \Ebig{  \Big| \intu  \frac{ w_j(x) (\hat{\phi}_j'-\phi'_j )(x) }{q(x) }  \d x\Big|^{k}}
\leq   C  \Big( C_{\Psi,2}\intu \frac{1}{q(x) }  \d x \Big)^{k-2} C_{\Psi,1} 
\intu  \frac{  \sum_{j=1}^n \Delta_j  |w_j(x)|^2  }{q(x)^2 }  \d x \\
=  &  C C_{\Psi,1} C_{\Psi,2}^{k-2} \Big( \intu \frac{1}{q(x) }  \d x \Big)^{k-1}. 
\end{align}

Combining this with \eqref{Rosen}, we have shown that for any $m\in \N$, there exists a  constant $C$ depending on $m$ and $\Dmax$ such that 
\begin{align}
 &\Ebig{ \Big|\intu \frac{ \epx -p(x)     }{ q(x) }  \d x \Big|^{2m} } 
\leq   C    \Big( C_{\Psi,1}  \intu \frac{1}{q(x)}  \d x       \Big)^m
\Big( \max\Big\{1, C_{\Psi,2}^2 \intu \frac{1}{q(x)}  \d x  \Big\}    
\Big)^{m-1} .  
\end{align}

-Consider now the second summand in \eqref{Zerlegung_Hauptlemma}. We apply the H\"older inequality, Lemma \ref{Lemma_Neumann} and Lemma \ref{Lemma_p} to derive that 
for a constant $C$ depending on $m$, $\DMax$  and $C_{4m}^{\pm}$, 
\begin{align}
 & \Ebig{  \Big|\intu ( \hat{p}(x) - p(x)  ) \Rem(x)  \d x  \Big|^{2m}           }
\leq \Big( \intu \ebig{ | \hat{p}(x) - p(x) |^{4m}}^{\frac{1}{4m}}   \ebig{|\Rem(x) |^{4m} }^{\frac{1}{4m}}  \d x  \Big)^{2m}\\
 \leq &  C \Big(  \intu   \frac{1}{q(x)} \d x \Big)^{2m}. 
\end{align}
-Finally, the third summand can be bounded as follows:  By the definition of $p$, the Cauchy-Schwarz inequality and  Lemma \eqref{Lemma_Neumann}, 
\begin{align}
 & \Ebig{ \Big|   \intu  p(x)  \Rem(x)  \d x \Big|^{2m}  }  \leq 
\Big( \intu |p(x) | \Ebig{|\Rem(x)|^{2m} }^{\frac{1}{2m} }  \d x \Big)^{2m} 
\leq C \Big( \intu \frac{|\Psi'(x)q(x) | }{q(x)^{\frac{3}{2}}  } \d x \Big)^{2m}   \\
\leq &C  \|\Psi'\|_{\lk^2}^{2m} \Big( \intu \frac{1}{q(x) }\d x \Big)^m. 
\end{align}

This completes the proof.   \hfill $\Box$

\subsection*{Acknowledgement.}  The author is grateful to Alexander Meister for intense discussions on this paper and for useful hints and ideas both, on the theoretical part and on the numerical procedure. 

\bibliographystyle{plainnat}
\bibliography{literatur}
\end{document}